\documentclass[11pt]{article} 

\usepackage{fourier} 
\usepackage{amsmath,amssymb,amsthm}
\usepackage{float}
\usepackage{fullpage}
\usepackage{subcaption} 
\usepackage{graphicx} 
\usepackage{authblk}

\usepackage{url} 
\usepackage[round]{natbib} 
\usepackage{color} 
\usepackage{physymb}

\usepackage[boxruled]{algorithm2e} 
\usepackage[colorlinks=true,citecolor=red,linkcolor=blue]{hyperref} 
\usepackage{natbib} 
\newtheorem{theorem}{Theorem}
\newtheorem{lemma}{Lemma}

\newtheorem{corollary}{Corollary}[theorem]

\theoremstyle{remark}
\newtheorem{remark}{Remark}

\theoremstyle{definition}
\newtheorem{definition}{Definition}
 
\newcommand{\R}{\mathbb{R}} 
 
\newcommand{\C}{\mathcal{C}} 
\newcommand{\mbb}[1]{\mathbb{#1}} \DeclareMathOperator*{\argmax}{arg\,max} 

\begin{document} 
\title{Inference via Message Passing on Partially Labeled Stochastic Block Models}

\date{}
\author[1]{T. Tony Cai\footnote{The research of Tony Cai was supported in part by NSF Grants DMS-1208982 and DMS-1403708,  and NIH Grant R01 CA127334.}}
\author[1]{Tengyuan Liang}
\author[1]{Alexander Rakhlin\footnote{Alexander Rakhlin gratefully acknowledges the support of NSF
under grant CAREER DMS-0954737.}}

\affil[1]{Department of Statistics, The Wharton School, University of Pennsylvania} 

\renewcommand\Authands{ and }

\maketitle 
\begin{abstract}
	We study the community detection and recovery problem in partially-labeled stochastic block models (SBM). We develop a fast linearized message-passing algorithm to reconstruct labels for SBM (with $n$ nodes, $k$ blocks, $p,q$ intra and inter block connectivity) when $\delta$ proportion of node labels are revealed. The signal-to-noise ratio ${\sf SNR}(n,k,p,q,\delta)$ is shown to characterize the fundamental limitations of inference via local algorithms. On the one hand, when ${\sf SNR}>1$, the linearized message-passing algorithm provides the statistical inference guarantee with mis-classification rate at most $\exp(-({\sf SNR}-1)/2)$, thus interpolating smoothly between strong and weak consistency. This exponential dependence improves upon the known error rate $({\sf SNR}-1)^{-1}$ in the literature on weak recovery. On the other hand, when ${\sf SNR}<1$ (for $k=2$) and ${\sf SNR}<1/4$ (for general growing $k$), we prove that local algorithms suffer an error rate at least $\frac{1}{2} - \sqrt{\delta \cdot {\sf SNR}}$, which is only slightly better than random guess for small $\delta$. 

\end{abstract}

\section{Introduction} 
The stochastic block model (SBM) is a well-studied model that addresses the clustering phenomenon in large networks. Various phase transition phenomena and limitations for efficient algorithms have been established for this ``vanilla'' SBM  \citep{coja2010graph,decelle2011asymptotic,massoulie2014community, mossel2012stochastic,mossel2013belief,krzakala2013spectral,abbe2014exact,hajek2014achieving,abbe2015community, deshpande2015asymptotic}. However, in real network datasets, additional side information is often available. This additional information may come, for instance, in the form of a small portion of revealed labels (or, community memberships), and this paper is concerned with methods for incorporating this additional information to improve recovery of the latent community structure.  Many global algorithms studied in the literature are based on spectral analysis (with belief propagation as a further refinement) or semi-definite programming. For these methods, it appears to be difficult to incorporate such additional side information, although some success has been reported \citep{cucuringu2012eigenvector,zhang2014phase}. Incorporating the additional information within local algorithms, however, is quite natural. In this paper, we focus on local algorithms and study their fundamental limitations. Our model is a {\bf partially labeled stochastic block model} (p-SBM), where $\delta$ portion of community labels are randomly revealed.

We address the following questions: 

\medskip
\noindent {\bf Phase Boundary}~~ 
Are there different phases of behavior in terms of the recovery guarantee, and what is the phase boundary for partially labeled SBM? How does the amount of additional information $\delta$ affect the phase boundary?

\medskip
\noindent {\bf Inference Guarantee}~~ What is the optimal guarantee on the recovery results for p-SBM and how does it interpolate between weak and strong consistency known in the literature? Is there an efficient and near-optimal parallelizable algorithm? 

\medskip
\noindent {\bf Limitation for Local v.s. Global Algorithms}~~  While optimal local algorithms (belief propagation) are computationally efficient, some global algorithms may be computationally prohibitive. Is there a fundamental difference in the limits for local and global algorithms? An answer to this question gives insights on the computational and statistical trade-offs.

\subsection{Problem Formulation} 
We define p-SBM with parameter bundle $(n,k,p,q,\delta)$ as follows. Let $n$ denote the number of nodes, $k$ the number of communities, $p$ and $q$ -- the intra and inter connectivity probability, respectively. The proportion of revealed labels is denoted by $\delta$. Specifically, one observes a partially labeled graph $G(V,E)$ with $|V| = n$, generated as follows. There is a latent disjoint partition $V=\bigcup_{l=1}^k V_l$ into $k$ equal-sized groups,\footnote{The result can be generalized to the balanced case, $|V_l| \asymp n/k$, see Section \ref{sec:local-tree}.} with $|V_l| = n/k$. The partition information introduces the latent labeling $\ell(v) = l$ iff $v \in V_l$. For any two nodes $v_i, v_j, 1\leq i,j\leq n$, there is an edge between them with probability $p$ if $v_i$ and $v_j$ are in the same partition, and with probability $q$ if not. Independently for each node $v \in V$, its true labeling is revealed with probability $\delta$. Denote the set of labeled nodes $V^{\rm l}$, its revealed labels $\ell(V^{\rm l})$, and unlabeled nodes by $V^{\rm u}$ (where $V = V^{\rm l} \cup V^{\rm u}$).

Equivalently, denote by $G \in \mathbb{R}^{n \times n}$ the adjacency matrix, and let $L \in \mathbb{R}^{n \times n}$ be the structural block matrix 
$$
L_{ij} = 1_{\ell(v_i)=\ell(v_j)},
$$ 
where $L_{ij} = 1$ iff node $i,j$ share the same labeling, $L_{ij}=0$ otherwise. 
Then we have independently for $1\leq i<j \leq n$
\begin{align*}
B_{ij} & \sim  {\sf Bernoulli}(p) \quad \text{if}~L_{ij}=1, \\
B_{ij} & \sim  {\sf Bernoulli}(q) \quad \text{if}~L_{ij}=0.
\end{align*}
Given the graph $G(V,E)$ and the partially revealed labels $\ell(V^{\rm l})$, we want to recover the remaining labels efficiently and accurately. We are interested in the case when $\delta(n),p(n),q(n)$ decrease with $n$, and $k(n)$ can either grow with $n$ or stay fixed. 

\subsection{Prior Work} 
In the existing literature on SBM without side information, there are two major criteria -- weak and strong consistency. Weak consistency asks for recovery better than random guessing in a sparse random graph regime ($p \asymp q \asymp 1/n$), and strong consistency requires exact recovery for each node above the connectedness theshold ($p \asymp q \asymp \log n/n $). Interesting phase transition phenomena in weak consistency for SBM have been discovered in \citep{decelle2011asymptotic} via insightful cavity method from statistical physics.  Sharp phase transitions for weak consistency have  been thoroughly investigated in \citep{coja2010graph,mossel2012stochastic,mossel2013belief,mossel2013proof,massoulie2014community}. In particular, spectral algorithms on the non-backtracking matrix have been studied in \citep{massoulie2014community} and the non-backtracking walk in \citep{mossel2013proof}. 
Spectral algorithms as initialization and belief propagation as further refinement to achieve optimal recovery was established in \citep{mossel2013belief}. The work of \cite{mossel2012stochastic} draws a connection between SBM thresholds and broadcasting tree reconstruction thresholds through the observation that sparse random graphs are locally tree-like. Recent work of \cite{abbe2015detection} establishes the positive detectability result down to the Kesten-Stigum bound for all $k$ via a detailed analysis of a modified version of belief propagation.
For strong consistency, \citep{abbe2014exact,hajek2014achieving,hajek2015achieving} established the phase transition using information theoretic tools and semi-definite programming (SDP) techniques. In the statistical literature, \cite{zhang2015minimax,gao2015achieving} investigated the mis-classification rate of the standard SBM.

\citep{kanade2014global} is one of the few papers that theoretically studied the partially labeled SBM. The authors investigated the stochastic block model where the labels for a vanishing fraction ($\delta \rightarrow 0$) of the nodes are revealed. The results focus on the asymptotic case when $\delta$ is sufficiently small and block number $k$ is sufficiently large, with no specified growth rate dependence. \cite{kanade2014global} show that pushing below the Kesten-Stigum bound is possible in this setting, connecting to a similar phenomenon in $k$-label broadcasting processes \citep{mossel2001reconstruction}. In contrast to these works, the focus of our study is as follows. Given a certain parameter bundle $\text{p-SBM}(n,k,p,q,\delta)$, we investigate the recovery thresholds as the fraction of labeled nodes changes, and determine the fraction of nodes that local algorithms can recover. 

The focus of this paper is on local algorithms. These methods, naturally suited for distributed computation \citep{linial1992locality}, provide efficient (sub-linear time) solutions to computationally hard combinatorial optimization problems on graphs. For some of these problems, they are good approximations to global algorithms. We refer to  \citep{kleinberg2000small} on the shortest path problem for small-world random graphs, \citep{gamarnik2014limits} for the maximum independent set problem for sparse random graphs, \citep{parnas2007approximating} on the minimum vertex cover problem, as well as \citep{nguyen2008constant}.

Finally, let us briefly review the literature on broadcasting processes on trees, from which we borrow technical tools to study p-SBM. Consider a Markov chain on an infinite tree rooted at $\rho$ with branching number $b$. Given the label of the root $\ell(\rho)$, each vertex chooses its label by applying the Markov rule $M$ to its parent's label, recursively and independently. The process is called broadcasting process on trees. One is interested in reconstructing the root label $\ell(\rho)$ given all the $n$-th level leaf labels.
Sharp reconstruction thresholds for the broadcasting process on general trees for the symmetric Ising model setting (each node's label is $\{+,-\}$) have been studied in \citep{evans2000broadcasting}. \cite{mossel2003information} studied a general Markov channel on trees that subsumes $k$-state Potts model and symmetric Ising model as special cases; the authors  established non-census-solvability below the Kesten-Stigum bound. \cite{janson2004robust} extended the sharp threshold to robust reconstruction cases, where the vertex' labels are contaminated with noise. 
In general, transition thresholds proved in the above literature correspond to the Kesten-Stigum bound $b |\lambda_2(M)|^2  = 1$ \citep{kesten1966limit,kesten1966additional}. We remark that for a  general Markov channel $M$, $b |\lambda_2(M)|^2  < 1$ does not always imply non-solvability --- even though it indeed implies non-census-solvability \citep{mossel2003information} --- which is equivalent to the extremality of free-boundary Gibbs measure. The non-solvability of the tree reconstruction problem below the Kesten-Stigum bound for a general Markov transition matrix $M\in \mathbb{R}^{k \times k}$ still remains open, especially for large $k$.

\subsection{Our Contributions}

This section summarizes the results. In terms of methodology, we propose a new efficient linearized message-passing Algorithm~\ref{alg:amp-psbm} that solves the label recovery problem of p-SBM in near-linear runtime. The algorithm shares the same transition boundary as the optimal local algorithm (belief propagation) and takes on a simple form of a weighted majority vote (with the weights depending on graph distance). This voting strategy is easy to implement (see Section~\ref{sec:num.study}). On the theoretical front, our focus is on establishing recovery guarantees  according to the size of the Signal-to-Noise Ratio ($\sf SNR$), defined as
\begin{align}
	\label{eq:key}
	{\sf SNR}(n,k,p,q,\delta):=(1-\delta) \frac{n(p-q)^2}{k^2(q + \frac{p-q}{k})}.
\end{align}

\medskip
 \noindent {\bf Phase Boundary}~~ For $k=2$, the phase boundary for recovery guarantee is  $${\sf SNR} = 1.$$ Above the threshold, the problem can be solved efficiently. Below the threshold, the problem is intrinsically hard. For growing $k$, on the one hand, a linearized message-passing algorithm succeeds when $${\sf SNR}  > 1,$$ matching the well-established Kesten-Stigum bound for all $k$. On the other hand, no local algorithms work significantly better than random guessing if $${\sf SNR}  < \frac{1}{4}.$$ 
	
\medskip
\noindent {\bf Inference Guarantee}~~ Above the ${\sf SNR}$ phase boundary, Algorithm~\ref{alg:amp-psbm}, a fast linearized message-passing algorithm $\hat{A}$ (with near-linear run-time $\mathcal{O}^*(n)$) provides near optimal recovery. For $k=2$, under the regime ${\sf SNR} > 1$, the proportion of mis-classified labels is at most $$\sup_{l \in \{+,-\}}~ \mbb{P}_l (\hat{A} \neq l) \leq \exp\left(-\frac{ {\sf SNR}-1 }{2 + o(1)}\right) \wedge \frac{1}{2}.$$ Thus when ${\sf SNR} \in (1, 2\log n)$, the recovery guarantee smoothly interpolates between weak and strong consistency. 
	On the other hand, below the boundary ${\sf SNR} < 1$, all local algorithms suffer the minimax classification error at least
	$$\inf_{\Phi} ~\sup_{l \in \{+,-\}}~ \mbb{P}_l (\Phi \neq l)  \geq \frac{1}{2} - \mathcal{O}\left(\sqrt{\frac{\delta}{1-\delta} \cdot \frac{{\sf SNR}}{1-{\sf SNR}}}\right).$$
	For growing $k$, above the phase boundary ${\sf SNR}>1$, the proportion of mis-classified labels is at most $$\sup_{l \in [k]}~ \mbb{P}_l (\hat{A} \neq l) \leq (k-1)\cdot \exp\left(-\frac{{\sf SNR} -1}{2+ o(1)}\right) \wedge \frac{k-1}{k}$$ via the approximate message-passing algorithm. However, below the boundary ${\sf SNR}<1/4$, the minimax classification error is lower bounded by
	$$
	\inf_{\Phi} ~\sup_{l \in [k]}~ \mbb{P}_l (\Phi \neq l)  \geq \frac{1}{2} - \mathcal{O}\left( \frac{\delta}{1-\delta}  \cdot \frac{ {\sf SNR} }{ 1 - 4 \cdot {\sf SNR} }  \vee \frac{1}{k} \right).
	$$

\medskip
\noindent {\bf Limitations of Local v.s. Global Algorithms}~~ It is known that the statistical boundary (limitation for global and possibly exponential time algorithms) for growing number of communities is ${\sf SNR} \asymp \mathcal{O}(\frac{\log k}{k})$ (\cite{abbe2015detection}, weak consistency) and ${\sf SNR} \asymp \mathcal{O}(\frac{\log n}{k})$ (\cite{chen2014statistical}, strong consistency). We show in this paper that the limitation for local algorithms (those that use neighborhood information up to depth $\log n$) is
	$$
	\frac{1}{4} \leq {\sf SNR} \leq 1 .
	$$
	In conclusion, as $k$ grows, \emph{there is a factor $k$ gap between the boundaries for global and local algorithms}. Local algorithms can be evaluated in near line time; however, the global algorithm achieving the statistical boundary requires exponential time.

\medskip

To put our results in the right context, let us make comparisons with the known literature. Most of the literature studies the standard SBM with no side labeling information. Here, many algorithms that achieve the sharp phase boundary are either global algorithms, or a combination of global and local algorithms, see \citep{mossel2013proof,massoulie2014community,hajek2014achieving,abbe2014exact}. However, from the theoretical perspective, it is not clear how to distinguish the limitation for global v.s. local algorithms through the above studies. In addition, from the model and algorithmic perspective, many global algorithms such as spectral \citep{coja2010graph,massoulie2014community} and semi-definite programming \citep{abbe2014exact,hajek2014achieving} are not readily applicable in a principled way when there is partially revealed labels. 

We try to resolve the above concerns. First, we establish a detailed statistical inference guarantee for label recovery.
Allowing for a vanishing $\delta$ amount of randomly revealed labels, we show that a fast local algorithm enjoys a good recovery guarantee that interpolates between weak and strong recovery precisely, down to the well-known Kesten-Stigum bound, for general $k$. The error bound $\exp(-({\sf SNR}-1)/2)$ proved in this paper improves upon the best known result of $({\sf SNR}-1)^{-1}$ in the weak recovery literature. We also prove that the limitation for local algorithms matches the Kesten-Stigum bound, which is sub-optimal compared to the limitation for global algorithms, when $k$ grows. We also remark that the boundary we establish matches the best known result for the standard SBM when we plug in $\delta=0$.   

We study the message-passing algorithms for multi-label broadcasting tree when a fraction of nodes' labels have been revealed. Unlike the usual asymptotic results for belief propagation and approximate message-passing, we prove \emph{non-asymptotic concentration of measure phenomenon} for messages on multi-label broadcasting trees. As the tree structure encodes detailed dependence among random variables, proving the concentration phenomenon requires new ideas. We further provide a lower bound on belief propagation for multi-label broadcasting trees.

\subsection{Organization of the Paper}

The rest of the paper is organized as follows. Section~\ref{sec:prelim} reviews the preliminary background and theoretical tools -- broadcasting trees -- that will be employed to solve the p-SBM problem. To better illustrate the main idea behind the theoretical analysis, we split the main result into two sections: Section~\ref{sec:k=2} resolves the recovery transition boundary for $k=2$, where the analysis is simple and best illustrates the main idea. In Section~\ref{sec:g-k}, we focus on the  growing $k = k(n)$ case, where a modified algorithm and a more detailed analysis are provided. In the growing $k$ case, we establish a distinct gap in phase boundaries between the global algorithms and local algorithms. 

\section{Preliminaries}
\label{sec:prelim}
 
\subsection{Broadcasting Trees}
First, we introduce the notation for the tree broadcasting process. Let $T_{\leq t}(\rho)$ denote the tree up to depth $t$ with root $\rho$. The collection of revealed labels for a broadcasting tree $T_{\leq t}(\rho)$ is denoted as $\ell_{T_{\leq t}(\rho)}$ (this is a collection of random variables). The labels for the binary broadcasting tree are $[2]:=\{+,-\}$ and for $k$-broadcasting tree $[k]:= \{ 1,2,\ldots,k \}$. For a node $v$, the set of labeled children is denoted by $\C^{\rm l}(v)$ and unlabeled ones by $\C^{\rm u}(v)$. We also denote the depth-$t$ children of $v$ to be $\C_{t}(v)$. For a broadcasting tree $T$, denote by $d$ its broadcasting number, whose rigorous definition is given in \citep{evans2000broadcasting,lyons2005probability}. For a broadcasting tree with bias parameter $\theta$, the labels are broadcasted in the following way: conditionally on the label of $v$,
\begin{equation*}
	\ell(u) = \left\{ \begin{array}{cl}
	\ell(v) & \text{w.p.}~~ \theta + \frac{1-\theta}{k} \\
	l \in [k]\backslash \ell(v) & \text{w.p.} ~~ \frac{1-\theta}{k} 
	 \end{array} \right. 
\end{equation*}
for any $u\in \C(v)$. In words, the child copies the color of its parent with probability $\theta + \frac{1-\theta}{k}$, or changes to any of the remaining $k-1$ colors with equal probability $\frac{1-\theta}{k}$.
For the node $v$, $N_{\C^{\rm l}(v)}(+)$ denotes the number of revealed positive nodes among its children. Similarly, we define $N_{\C^{\rm l}(v)}(l)$ for $l \in [k]$ in multi-label trees.

\subsection{Local Tree-like Graphs \& Local Algorithms}
\label{sec:local-tree}
When viewed locally, stochastic block models share many properties with broadcasting trees. In fact, via the coupling lemma (see Lemma~\ref{lma:coupling.tree}) from \citep{mossel2012stochastic}, one can show the graph generated from the stochastic block model is locally a tree-like graph. For the rest of the paper, we abbreviate the following maximum coupling depth $\bar{t}_{n,k,p,q}$ as $\bar{t}$ (see Lemma~\ref{lma:coupling.tree} for details). 
\begin{definition}[$\bar{t}$-Local Algorithm Class for p-SBM]
	 \label{def:local.alg}
	 The $\bar{t}$-local algorithm class is the collection of decentralized algorithms that run in parallel on nodes of the graph. To recover a node $v$'s label in p-SBM, an algorithm may only utilize information (revealed labels, connectivity) of the local tree $T_{\leq \bar{t}}(v)$ rooted at $v$ with depth at most $\bar{t}$.
\end{definition}
In view of the coupling result, for the stochastic block model $\text{p-SBM}(n,k=2,p,q,\delta)$, as long as we focus on $\bar{t}$-local algorithms, we can instead study the binary-label broadcasting process $\text{Tree}_{k=2}(\theta,d,\delta)$ with broadcasting number $d = \frac{n}{2}(p+q)$ and bias parameter $\theta = \frac{p-q}{p+q}$. Similarly, for the multi-label model $\text{p-SBM}(n,k,p,q,\delta)$, we will study the $k$-label broadcasting process $\text{Tree}_k(\theta, d, \delta)$ with broadcasting number $d = n(q + \frac{p-q}{k})$ and bias parameter $\theta = \frac{p-q}{k(q + \frac{p-q}{k})}$. \footnote{In the balanced  SBM case, for each node, the local tree changes slightly with different branching number and bias parameter.} For each layer of the broadcasting tree, $\delta$ portion of nodes' labels are revealed. Our goal is to understand the condition under which message-passing algorithms on multi-label broadcasting trees  succeed in recovering the root label.

\subsection{Hyperbolic Functions and Other Notation} 
In order to introduce the belief propagation and message-passing algorithms, let us recall several hyperbolic functions that will be  used frequently. As we show, linearization of the hyperbolic function induces a new approximate message-passing algorithm. Recall that
\begin{align*}
	\tanh x = \frac{e^{x} - e^{-x}}{e^{x} + e^{-x}}, \quad \atanh x = \frac{1}{2} \log \left( \frac{1+x}{1-x} \right), 
\end{align*}
and define 
\begin{align}
	\label{eq:f_theta}
	f_{\theta}(x)& := 2 \atanh \left( \theta \tanh \frac{x}{2} \right) = \log \frac{1+\theta \cdot \frac{e^x -1 }{e^x + 1}}{1- \theta \cdot \frac{e^x -1 }{e^x + 1}}. 
\end{align}
\begin{figure}[H] 
	\centering 
	\includegraphics[width=3.5in]{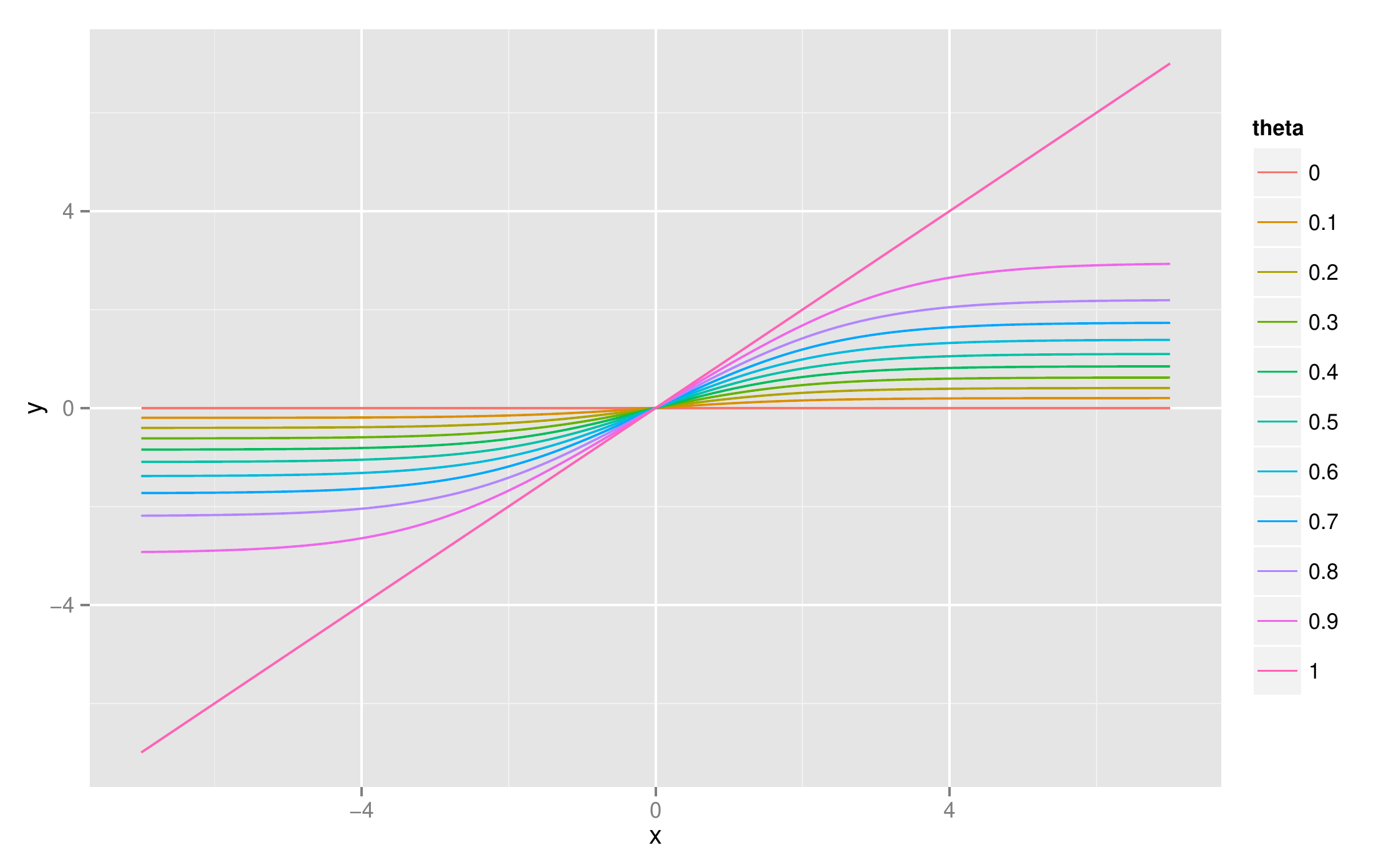} \caption{Function $f_{\theta}$ for $\theta\in[0,1]$.} 
	\label{fig:hyperbolic}
\end{figure}
The function $f_{\theta}: \R \rightarrow \R$ is a contraction with $$ |f(x) - f(y)| \leq \theta |x - y|$$ since
\begin{align*}
	\frac{d f_{\theta}(x)}{d x} = \frac{2 \theta}{(1-\theta^2) \cosh(x) + (1 + \theta^2) } \leq \theta.
\end{align*}
An illustration of $f_\theta$ is provided in Figure~\ref{fig:hyperbolic}. The recursion rule for message passing can be written succinctly using the function $f_\theta$, as we show in Section~\ref{sec:bp.amp.k=2}.

Let us collect a few remaining definitions. The moment generating function (MGF) for a random variable $X$ is denoted by  $\Psi_X(\lambda) = \mbb{E} e^{\lambda X}$, for $\lambda>0$, and the cumulant generating function is defined as $K_X(\lambda) = \log \Psi_X(\lambda)$. For asymptotic order of magnitude, we use $a(n) = \mathcal{O}(b(n))$ to mean $\forall n, a(n) \leq C b(n)$ for some universal constant $C$, and use $\mathcal{O}^*(\cdot)$ to omit the poly-logarithmic dependence. As for notation $\precsim, \succsim$: $a(n) \precsim b(n)$ if and only if $\varlimsup \limits_{n\rightarrow \infty} \frac{a(n)}{b(n)} \leq c$, with some constant $c>0$, and vice versa. The square bracket $[\cdot]$ is used to represent the index set $[k]:=[1,2,\ldots,k]$; in particular when $k=2$, $[2]:=\{+,-\}$ for convenience.

\section{Number of Communities $k=2$~: Message Passing with Partial Information}
\label{sec:k=2}

\subsection{p-SBM Transition Thresholds}
\label{sec:k=2.thresholds}

We propose a novel linearized message-passing algorithm to solve the p-SBM in near-linear time. The method employs  Algorithm~\ref{alg:AMP} and \ref{alg:k-AMP} as sub-routines, can run in parallel, and is easy to implement.

\begin{algorithm}[H] 
	\KwData{A network graph $G(V,E)$ with partial label information, where $V = V^{\rm l} \cup V^{\rm u}$ is composed of labeled set and unlabeled set. Denote $\epsilon = o(1)$ small, and $\bar{t} \precsim \frac{\log n}{\log (n(p+q))}$.} 
	\KwResult{The labeling for each node $v \in V^{\rm u}$.} 
	\For{each node $v \in V^{\rm u}$ in the unlabeled set,}{open the tree neighborhood $T_{\leq \bar{t}}(v)$ induced by the graph $G(V,E)$ \;
		\For {each node $u \in \C_{(1-\epsilon)\bar{t}}(v)$, i.e., depth $(1-\epsilon)\bar{t}$ child of $v$,}{focus on the subtree $T_{\leq \epsilon \bar{t}}(u)$,  \; initialize the message for $u$ via the labeled node $\in V^{\rm l}$ in layer $\epsilon \bar{t}$ of the subtree \footnote{as there is at least one labeled node in layer $\epsilon \bar{t}$, for all the subtrees rooted in $u$}  \;} 
	 	run message-passing Algorithm~\ref{alg:AMP} (Algorithm~\ref{alg:k-AMP} for general $k$) on the tree $T_{\leq (1-\epsilon)\bar{t}}(v)$ with initial message on layer $(1-\epsilon)\bar{t}$ \;
		output $\ell(v)$. 
		} 
	
	\caption{Message Passing for p-SBM} 
	\label{alg:amp-psbm} 
\end{algorithm}
\medskip

Now we are ready to present the main result.

\begin{theorem}[Transition Thresholds for p-SBM: $k=2$]
	\label{thm:sbm.k=2}
	Consider the partially labeled stochastic block model $G(V,E)$ and its revealed labels $\ell(V^{l})$ under the conditions (1) $np \asymp nq \precsim n^{o(1)}$ and (2) $\delta \succsim n^{-o(1)}$. For any node $\rho \in V^{u}$ and its locally tree-like neighborhood $T_{\leq \bar{t}}(\rho)$, define the maximum mis-classification error of a local estimator ~$\Phi: \ell_{T_{\leq \bar{t}}(\rho)} \rightarrow \{ +, - \}$ as
	\begin{align*}
		{\sf Err}(\Phi) := \max_{l \in \{+,- \}}~\mbb{P} \left( \Phi( \ell_{T_{\leq t}(\rho) }) \neq \ell(\rho) |  \ell(\rho) = l  \right) .
	\end{align*}
	The transition boundary for p-SBM depends on the value  
	\begin{align*}
		{\sf SNR} = (1-\delta) \frac{n(p-q)^2}{2(p+q)}.
	\end{align*} 
	($k=2$ in Eq.~\eqref{eq:key}).
	On the one hand, if
	\begin{align}
		\label{eq:upp.k=2}
		{\sf SNR} > 1,
	\end{align}
	the $\bar{t}$- local message-passing Algorithm~\ref{alg:amp-psbm} --- denoted as $\hat{A}(\ell_{T_{\leq \bar{t}}(
	\rho)})$ --- recovers the true labels of the nodes with mis-classification rate at most 
	\begin{align}
		\label{eq:percent.k=2}
	{\sf Err}(\hat{A})  \leq   \exp\left(- \frac{ {\sf SNR} - 1}{2 C + o_{\bar{t}}(1)}  \right) \wedge \frac{1}{2}, 
	\end{align}
	where $C>0$ is a constant and $C\equiv 1$ if the local tree is regular. On the other hand, when
	\begin{align}
		\label{eq:low.k=2}
		{\sf SNR} < 1,
	\end{align}
	for any $\bar{t}$-local estimator $\Phi: \ell_{T_{\leq \bar{t}}(\rho)} \rightarrow \{ +, - \}$, the minimax mis-classification error is lower bounded as
	$$
	\inf_{\Phi}~~  {\sf Err}(\Phi) \geq \frac{1}{2} - C \cdot \sqrt{\frac{\delta}{1-\delta} \cdot \frac{{\sf SNR}}{1 - {\sf SNR}} \cdot \frac{(p+q)^2}{pq}}  = \frac{1}{2} - C' \cdot \sqrt{\frac{\delta}{1-\delta} \cdot \frac{{\sf SNR}}{1 - {\sf SNR}}}.
	$$
\end{theorem}
	The above lower bound in the regime $\delta = o(1)$ implies that no local algorithm using information up to depth $\bar{t}$ can do significantly better than $1/2 + \mathcal{O}(\sqrt{\delta})$, close to random guessing.

Let us compare the main result for p-SBM with the well-known result for the standard SBM with no partial label information. The boundary in Equations~\eqref{eq:upp.k=2} and \eqref{eq:low.k=2} is the phase transition boundary for the standard SBM when we plug in $\delta=0$. This also matches the well-known Kesten-Stigum bound. For the standard SBM in $k=2$ case, the Kesten-Stigum bound is proved to be sharp (even for global algorithms), see \citep{mossel2013proof, massoulie2014community}.

The interesting case is when there is a vanishing amount of revealed label information, i.e., $o(1) = \delta \succsim n^{-o(1)}$. In this case, the upper bound part of Theorem~\ref{thm:sbm.k=2} states that this vanishing amount of initial information is enough to propagate the labeling information to all the  nodes, above the same detection transition threshold as the vanilla SBM. However, the theoretical guarantee for the label propagation pushes beyond weak consistency (detection), explicitly interpolating between weak and strong consistency. The result provides a detailed understanding of the strength of the ${\sf SNR}$ threshold and its effect on percentage recovery guarantee, i.e., the inference guarantee. More concretely, for the regime $p = a/n, q = b/n$, the boundary 
$$
{\sf SNR} =   (1-\delta) \frac{n(p-q)^2}{2(p+q)} >1
$$
which is equivalent to the setting
$\frac{(a-b)^2}{2(a+b)}> \frac{1}{1-\delta}.$
When $\delta=0$, this matches the boundary for weak consistency  in  \citep{mossel2013proof, massoulie2014community}. In addition, ${\sf SNR} > 1+ 2\log n$ implies 
${\sf Err}(\hat{A}) < 1/n \rightarrow 0$, which means strong consistency (recovery) in the regular tree case ($C \equiv 1$). This condition on ${\sf SNR}$ is satisfied, for instance, by taking $p = a\log n/n, q = b\log n/n$ and 
computing the relationship between $a, b$, and $\delta$ to ensure
$$ {\sf SNR} = (1-\delta) \frac{n(p-q)^2}{2(p+q)} > 1+ 2\log n.$$
This relationship is precisely
$$ \frac{\sqrt{a} - \sqrt{b}}{\sqrt{2}} > \sqrt{\frac{1+\frac{1}{2\log n}}{1-\delta}} \cdot \frac{\sqrt{a}+\sqrt{b}}{\sqrt{2(a+b)}} \succsim \sqrt{\frac{1}{1-\delta}}.$$
The above agrees with the scaling for strong recovery in \citep{abbe2014exact,hajek2014achieving}.

\bigskip

The following sections are dedicated to proving the theorem. The upper bound is established in Corollary~\ref{col:amp.k=2} through a linearized belief propagation that serves as a subroutine for Algorithm~\ref{alg:amp-psbm}. The lower bound is established by  employing the classic Le Cam's theory, as shown in Theorem~\ref{thm:low.k=2}.

\subsection{Belief Propagation \& Message Passing}
\label{sec:bp.amp.k=2}

In this section we introduce the belief propagation (BP) Algorithm~\ref{alg:BP} and motivate the new message-passing Algorithm~\ref{alg:AMP} that, while being easier to analyze, mimics the behavior of BP. Algorithm~\ref{alg:AMP} serves as the key building block for Algorithm~\ref{alg:amp-psbm}.

Recall the definition of the partially revealed binary broadcasting tree $\text{Tree}_{k=2}(\theta,d,\delta)$ with broadcasting number $d$. The root $\rho$ is labeled $\ell(\cdot)$ with either $\{ +, - \}$ equally likely, and the label is not revealed. The labels are broadcasted along the tree with a bias parameter $0<\theta<1$: for a child $v \in \C(u)$ of $u$, $\ell(v)=\ell(u)$ with probability $\frac{1+\theta}{2}$ and $\ell(v)=-\ell(u)$ with probability $\frac{1-\theta}{2}$.
The tree is partially labeled in the sense that a fraction $0<\delta<1$ of labels are revealed for each layer and  $\ell_{T_{\leq t}(\rho)}$ stands for the revealed label information of tree rooted at $\rho$ with depth $\leq t$.

Let us formally introduce the BP algorithm, which is the Bayes optimal algorithm on trees. We define $$M_{i}(\ell_{T_{\leq i}(v)}):= \log \frac{\mbb{P}\left( \ell(v) = + | \ell_{T_{\leq i}(v)} \right)}{\mbb{P}\left( \ell(v) = - | \ell_{T_{\leq i}(v)} \right)}$$
as the belief of node $v$'s label, and we abbreviate it as $M_i$ when the context is clear. The belief depends on the revealed information $\ell_{T_{\leq i}(v)}$.
The following Algorithm~\ref{alg:BP} calculates the log ratio $M_{t}(\ell_{T_{\leq t}(\rho)})$ based on the revealed labels up to depth $t$, recursively, as shown in Figure~\ref{fig:bp}. The Algorithm is derived through Bayes' rule and simple algebra, and the detailed derivation is included in Section~\ref{pf:deri.k=2}.

\begin{figure}[pht] 
	\centering 
	\includegraphics[width=3in]{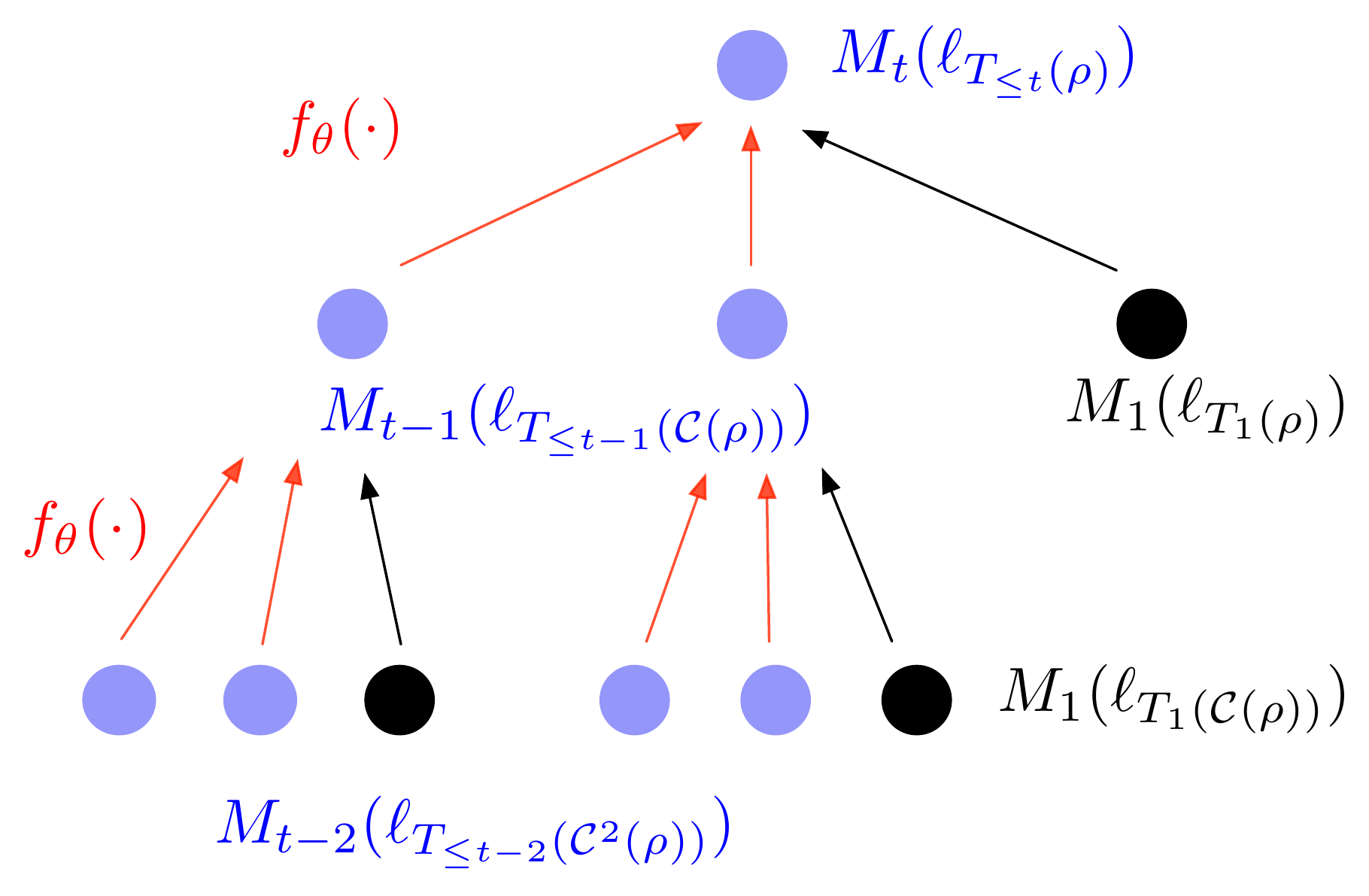} 
	\caption{Illustration of recursion  in Eq.~\eqref{eq:update-message} for messages on a $d$-regular tree. Here $d=3$ with two unlabeled children ($(1-\delta)d = 2$, denoted by blue) and one labeled child ($\delta d=1$, denoted by black),  and the depth is $2$. $\C^{t}(\rho)$ denotes depth $t$ children of the root $\rho$. The red arrows correspond to messages received from the labeled children and black arrow are from the unlabeled children. } 
	\label{fig:bp}
\end{figure}

\begin{algorithm}[H] 
	\KwData{A partially labeled tree $T_{\leq t}(\rho)$ with depth $t$, with labels $\ell_{T_{\leq t}(\rho)}$, the root label $\ell(\rho)$ is unknown.} 
	\KwResult{The logit of the posterior probability $M_{t}(\ell_{T_{\leq t}(\rho)}) = \log \frac{\mbb{P}\left( \ell(\rho) = + | \ell_{T_{\leq t}(\rho)} \right)}{\mbb{P}\left( \ell(\rho) = - | \ell_{T_{\leq t}(\rho)} \right)}$.} Initialization: $i = 1$, and $M_{0}(\ell_{T_{\leq 0}(v)}) = 0$, $M_{1}(\ell_{T_{\leq 1}(v)}) = \left(N_{\C^{\rm l}(v)}(+) - N_{\C^{\rm l}(v)}(-) \right) \log \frac{1+\theta}{1-\theta}$, $\forall v \in T_{\leq t}(\rho) $\; 
	\While{$i \leq t$}{ focus on $(t-i)$-th layer\; \For{$v \in \C_{t-i}(\rho)$ and $v$ unlabeled}{ update messages for the subtree: 
	\begin{align}
		\label{eq:update-message}
	M_{i}(\ell_{T_{\leq i}(v)}) = M_{1}(\ell_{T_1(v)}) + \sum\limits_{u \in \C^{\rm u}(v)} f_{\theta}\left(M_{i-1}(\ell_{T_{\leq i-1}(u)})\right)
	\end{align} move one layer up: $i = i+1$ \; } } 
	output $M_{t}(\ell_{T_{\leq t}(\rho)})$. 
	\caption{Belief Propagation (BP) on Partially Labeled Binary Broadcasting Tree} \label{alg:BP} 
\end{algorithm}
\medskip
The computational complexity of this algorithm is $$\mathcal{O} \left( \frac{(\delta d +1)[(1-\delta)d]^t - d}{(1-\delta)d - 1} \right).$$ While the method is Bayes optimal, the density of the messages $M_i$ is difficult to analyze, due to the dependence on revealed labels and the non-linearity of $f_{\theta}$. However, the following linearized version, Algorithm~\ref{alg:AMP}, shares many theoretical similarities with Algorithm~\ref{alg:BP}, and is easier to analyze. Both Algorithms~\ref{alg:BP},~\ref{alg:AMP} require the prior knowledge of $\theta$.\\

\begin{algorithm}[H] 
	\KwData{A partially labeled tree $T_{\leq t}(\rho)$ with depth $t$, with labels $\ell_{T_{\leq t}(\rho)}$, the root label $\ell(\rho)$ is unknown.} 
	\KwResult{Label $\ell(\rho) = \sgn(M_{t}(\ell_{T_{\leq t}(\rho)}))$.} Initialization: $i = 1$, and $M_{0}(\ell_{T_{\leq 0}(v)}) = 0$, $M_{1}(\ell_{T_{\leq 1}(v)}) = \left(N_{\C^{\rm l}(v)}(+) - N_{\C^{\rm l}(v)}(-) \right) \log \frac{1+\theta}{1-\theta}$, $\forall v \in T_{\leq t}(\rho) $\; 
	\While{$i \leq t$}{ focus on $(t-i)$-th layer\; \For{$v \in \C_{t-i}(\rho)$ and $v$ unlabeled}{ update messages for the subtree: 
	\begin{align}
		M_{i}(\ell_{T_{\leq i}(v)}) = M_{1}(\ell_{T_1(v)}) + \theta \cdot \sum\limits_{u \in \C^{\rm u}(v)} M_{i-1}(\ell_{T_{\leq i-1}(u)})
	\end{align}
	move one layer up: $i = i+1$ \; } } output $\ell(\rho) = \sgn(M_{t}(\ell_{T_{\leq t}(\rho)}))$. \caption{Approximate Message Passing (AMP) on Partially Labeled Binary Broadcasting Tree} \label{alg:AMP} 
\end{algorithm}
\medskip

 Algorithm~\ref{alg:AMP} can also be viewed as a weight-adjusted majority vote algorithm. We will prove in the next two sections that BP and AMP achieve the same transition boundary in the following sense. Above a certain threshold, the AMP algorithm succeeds, which implies that the optimal BP algorithm will also work. Below the same threshold, even the optimal BP algorithm will fail, and so does the AMP algorithm. 


\subsection{Concentration Phenomenon on Messages} 

We now prove Theorem~\ref{thm:AMP.k=2}, which shows the concentration of measure phenomenon for messages defined on the broadcasting tree. We focus on a simpler case of regular local trees, and the result will be generalized to Galton-Watson trees with a matching branching number. 

We state the result under a stronger condition $\delta d  = O(1)$. In the case when $\delta d = o(1)$, a separate trick, described in Remark~\ref{rem:sparse} below, of aggregating the $\delta$ information in a subtree will work.
\begin{theorem}[Concentration of Messages for AMP] 
	\label{thm:AMP.k=2} 
	Consider the Approximate Message Passing (AMP) Algorithm~\ref{alg:AMP} on the binary-label broadcasting tree $\text{Tree}_{k=2}(\theta,d,\delta)$. Assume $\delta d = O(1)$. Define parameters $\{\mu_t$, $\sigma_t^2\}_{t \geq 0}$ as 
	\begin{align}
		\label{eq: mean-var-update}
		\mu_t & = \mu_1 + \alpha \cdot \mu_{t-1}, \\
		\sigma_t^2 & = \sigma_1^2 + \alpha \cdot \sigma_{t-1}^2 + \alpha \cdot \mu_{t-1}^2,
	\end{align}
	with the initialization 
	\begin{align*}
		\mu_1 & = \theta \delta d \cdot \log \frac{1+\theta}{1-\theta}, \quad \sigma_1^2 = \delta d \cdot \log^2\frac{1+\theta}{1-\theta}, \\
		\alpha & : = (1-\delta)\theta^2 d.
	\end{align*}
	The explicit formulas for $\mu_t$ and $\sigma^2_t$ are
	\begin{align}
		& \mu_t = \frac{\alpha^t-1}{\alpha-1} \cdot \mu_1, \\
		& \sigma_t^2 = \frac{\alpha^t-1}{\alpha-1} \cdot \sigma_1^2 + \frac{\frac{\alpha^{2t}-\alpha^{t+1} + \alpha^{t} - \alpha}{\alpha-1} - 2(t-1)\alpha^t}{(\alpha-1)^2} \cdot \mu_1^2. 
	\end{align}
	For a certain depth $t$, conditionally on $\ell(\rho) = +$, the messages in Algorithm~\ref{alg:AMP} concentrate as
	\begin{align*}
		\mu_t - x \cdot \sigma_t \leq M_t(\ell_{T_{\leq t}(\rho)}) \leq \mu_t + x \cdot \sigma_t, 
	\end{align*}
	and conditionally on $\ell(\rho) = -$, 
	\begin{align*}
		- \mu_t - x \cdot \sigma_t \leq M_t(\ell_{T_{\leq t}(\rho)}) \leq - \mu_t + x \cdot \sigma_t, 
	\end{align*}
	both with probability at least $1-2\exp(x^2/2)$. 
\end{theorem}

Using Theorem~\ref{thm:AMP.k=2}, we establish the following positive result for approximate message-passing. 
\begin{corollary}[Recovery Proportions for AMP, $\alpha>1$] 
	\label{col:amp.k=2}
	Assume $$\alpha:=(1-\delta)\theta^2 d > 1,$$ and for any $t$ define 
	\begin{align*}
		\epsilon(t) = \frac{(\alpha-1)^2}{\theta^2\delta d} \frac{1}{\alpha^t-1} + \mathcal{O}(\alpha^{-t}), ~~ \text{with}~~ \lim_{t \rightarrow \infty} \epsilon(t) = 0. 
	\end{align*}
	Algorithm~\ref{alg:AMP} recovers the label of the root node with probability at least $$ 1 - \exp\left( - \frac{\alpha - 1}{2(1 + \epsilon(t))} \right), $$
	and its computational complexity is $$ \mathcal{O} \left( \frac{(\delta d +1)[(1-\delta)d]^t - d}{(1-\delta)d - 1} \right).$$ 
\end{corollary}

\begin{remark}
	\label{rem:sparse}
For the sparse case $\delta d = o(1)$, we employ the following technique. Take $t_0>0$ to be the smallest integer such that $\delta [(1-\delta)d]^{t_0} > 1$. For each leaf node $v$, open a depth $t_0$ subtree rooted at $v$, with the number of labeled nodes ${\sf Poisson}(\delta [(1-\delta)d]^{t_0})$. Then we have the following parameter updating rule
	\begin{align*}
		\mu_t  = \alpha \cdot \mu_{t-1}, \quad \sigma_t^2  = \alpha \cdot \sigma_{t-1}^2 + \alpha \cdot \mu_{t-1}^2,
	\end{align*}
	with initialization 
	\begin{align*}
		\mu_1 & = \theta^{t_0}  \cdot \log \frac{1+\theta}{1-\theta}, \quad \sigma_1^2 =  \log^2\frac{1+\theta}{1-\theta}, \\
		\alpha & : = (1-\delta)\theta^2 d.
	\end{align*}
The explicit formulas for $\mu_t$ and $\sigma^2_t$ based on the above updating rules are 
	\begin{align*}
		& \mu_t = \alpha^{t-1} \cdot \mu_1,  \quad \sigma_t^2 = \alpha^{t-1} \cdot \sigma_1^2 + \frac{\alpha^{t-1}(\alpha^{t-1}-1)}{\alpha-1} \cdot \mu_1^2. 
	\end{align*}
	Corollary~\ref{col:amp.k=2} will change as follows: the value $\epsilon(t)$ is now
		\begin{align*}
			\epsilon(t) = \frac{1}{\theta^{2t_0}} \frac{1}{\alpha^{t-1}}, ~~ \text{with}~~ \lim_{t \rightarrow \infty} \epsilon(t) = 0. 
		\end{align*}
		This 
		slightly modified algorithm recovers the label of the root node with probability at least $ 1 - \exp\left( - \frac{\alpha - 1}{2(1 + \epsilon(t))} \right). $
\end{remark}

\subsection{Lower Bound for Local Algorithms: Le Cam's Method}
\label{sec:lower}
In this section we show that the $\sf SNR$ threshold in Theorem~\ref{thm:sbm.k=2} and Corollary~\ref{col:amp.k=2} is sharp for all local algorithms. The limitation for local algorithms is proved along the lines of Le Cam's method. If we can show a small upper bound on total variation distance between two tree measures $\mu_{\ell_{\leq t}(+)}, \mu_{\ell_{\leq t}(-)}$, then no algorithm utilizing the information on the tree can distinguish these two measures well. Theorem~\ref{thm:low.k=2} formalizes this idea. 

\begin{theorem}[Limits of Local Algorithms]
	\label{thm:low.k=2}
	Consider the following two measures of revealed labels defined on trees: $\mu^{+}_{\ell_{T_{\leq t}(\rho)}},  \mu^{-}_{\ell_{T_{\leq t}(\rho)}}$. Assume that $\delta d >1$, $
	(1-\delta)\theta^2d<1,
	$ and $2 \delta d \log \left( 1 + \frac{4\theta^2}{1 - \theta^2} \right)< [1-(1-\delta)\theta^2 d]^2$. Then for any $t>0$, the following bound on total variation holds
	\begin{align*}
	d_{\rm TV}^2 \left( \mu^{+}_{\ell_{T_{\leq t}(\rho)}} , \mu^{-}_{\ell_{T_{\leq t}(\rho)}} \right)  \leq   \frac{2 \delta d \log \left( 1 + \frac{4\theta^2}{1 - \theta^2} \right)}{1 - (1-\delta)\theta^2 d}.
	\end{align*}
	The above bound implies
	$$
	\inf_{\Phi} \sup_{l(\rho) \in \{+,-\}} \mbb{P} \left(\Phi(\ell_{T_{\leq t}(\rho)}) \neq \ell(\rho) \right) \geq \frac{1}{2} - C \cdot  \left\{  \frac{\delta d \log \left( 1 + \frac{4\theta^2}{1 - \theta^2} \right)}{1 - (1-\delta)\theta^2 d} \right\}^{1/2},
	$$
	where~ $\Phi: \ell_{T\leq t}(\rho) \rightarrow \{ +,- \}$ is any estimator mapping the revealed labels in the local tree to a decision, and $C>0$ is some universal constant.
\end{theorem}

We defer the proof of the Theorem~\ref{thm:low.k=2} to Section~\ref{sec:proof}. Theorem~\ref{thm:low.k=2} assures the optimality of Algorithm~\ref{alg:AMP}.

\section{Growing Number of Communities} 
\label{sec:g-k}

In this section, we extend the algorithmic and theoretical results to p-SBM with general $k$. There is a distinct difference between the case of large $k$  and $k=2$: there is a factor gap between the boundary achievable by local and global algorithms. 

The main Algorithm that solves p-SBM for general $k$ is still Algorithm~\ref{alg:amp-psbm}, but this time it takes Algorithm~\ref{alg:k-AMP} as a subroutine. We will first state  Theorem~\ref{thm:sbm.g-k}, which summarizes the main result.

\subsection{p-SBM Transition Thresholds}

	The transition boundary for partially labeled stochastic block model depends on the critical value ${\sf SNR}$
	defined in Equation~\eqref{eq:key}.
	
\begin{theorem}[Transition Thresholds for p-SBM: general $k$]
	\label{thm:sbm.g-k}
	Assume (1) $np \asymp nq \precsim n^{o(1)}$, (2) $\delta \succsim n^{-o(1)}$, (3) $k \precsim n^{o(1)}$, and consider the partially labeled stochastic block model $G(V,E)$ and the revealed labels $\ell(V^{l})$. For any node $\rho \in V^{u}$ and its locally tree-like neighborhood $T_{\leq \bar{t}}(\rho)$, define the maximum mis-classification error for a local estimator ~$\Phi: \ell_{T_{\leq \bar{t}}(\rho)} \rightarrow [k]$ as
	\begin{align*}
		{\sf Err}(\Phi) := \max_{l \in [k]}~\mbb{P} \left( \Phi( \ell_{T_{\leq t}(\rho) }) \neq \ell(\rho) |  \ell(\rho) = l  \right) .
	\end{align*}	
 On the one hand, if
	\begin{align}
		\label{eq:upp.gk}
		{\sf SNR} > 1,
	\end{align}
	the $\bar{t}$- local message-passing Algorithm~\ref{alg:amp-psbm}, denoted by $\hat{A}(\ell_{T_{\leq \bar{t}}(
	\rho)})$, recovers the true labels of the nodes, with mis-classification rate at most 
	\begin{align}
		\label{eq:percent.g-k}
	{\sf Err}(\hat{A})  \leq  (k-1) \exp\left(- \frac{ {\sf SNR} - 1}{2 C + o_{\bar{t}}(1)}  \right) \wedge \frac{k-1}{k}, 
	\end{align}
	where $C\equiv 1$ if the local tree is regular. On the other hand, if
	\begin{align}
	\label{eq:low.gk}
		{\sf SNR} < \frac{1}{4},
	\end{align}
	for any $\bar{t}$-local estimator $\Phi: \ell_{T_{\leq \bar{t}}(\rho)} \rightarrow [k]$, the minimax mis-classification error is lower bounded as
	$$
	\inf_{\Phi}~~  {\sf Err}(\Phi) \geq \frac{1}{2} \left( 1 - C \cdot \frac{\delta}{1-\delta}  \cdot \frac{ {\sf SNR} }{ 1 - 4 \cdot {\sf SNR} } \cdot \frac{(p+q)(q+(p-q)/k)}{pq} - \frac{1}{k} \right) > \frac{1}{2} - C' \frac{\delta}{1-\delta}  \cdot \frac{ {\sf SNR} }{ 1 - 4 \cdot {\sf SNR} } \vee \frac{1}{k},
	$$
	where $C = C' \equiv 1$ if the local tree is regular.

\end{theorem}
When $\delta = o(1)$ and $k>2$, the above lower bound says that no local algorithm (that uses information up to depth $\bar{t}$) can consistently estimate the labels with vanishing error.

	As we did for $k=2$, let us compare the main result for p-SBM with the well-known result for the standard SBM with no partial label information. The boundary in Equation~\eqref{eq:upp.gk} matches the detection bound in \citep{abbe2015detection} for standard SBM when we plug in $\delta=0$, which also matches the well-known Kesten-Stigum (K-S) bound. In contrast to the case $k=2$, it is known that the K-S bound is not sharp when $k$ is large, i.e., there exists an algorithm which can succeed below the K-S bound. A natural question is whether K-S bound is sharp within a certain local algorithm class. As we show in Equation~\eqref{eq:low.gk}, below a quarter of the K-S bound, the distributions (indexed by the root label) on the revealed labels for the local tree are bounded in the total variation distance sense, implying that no local algorithm can significantly push below the K-S bound.  In summary, $1/4 \leq {\sf SNR} \leq 1$ is the limitation for local algorithms. Remarkably, it is known in the literature \citep{chen2014statistical,abbe2015detection} that information-theoretically the limitation for global algorithms is ${\sf SNR} = \mathcal{O}^*(1/k)$.  This suggests a possible computational and statistical gap as $k$ grows.

\subsection{Belief Propagation \& Message Passing}

In this section, we investigate the message-passing Algorithm~\ref{alg:k-AMP} for p-SBM with $k$ blocks, corresponding to multi-label broadcasting trees. Denote $X_t^{(i)}(\ell_{T_{\leq t}(v)}) = \mbb{P} \left( \ell(v) = i | \ell_{T_{\leq t} (v)} \right)$. For $u \in \C(v)$, 
\begin{align*}
	\mbb{P}\left(\ell(u)= \ell(v) | \ell(v) \right) &= \theta + \frac{1-\theta}{k} \\
	\mbb{P}\left(\ell(u)= l \in [k]\backslash \ell(v) | \ell(v)\right) &= \frac{1-\theta}{k}. 
\end{align*}
For any $j \neq i \in [k]$ and general $t$, the following Lemma describes the recursion arising from the Bayes theorem.
\begin{lemma}
	\label{lma:k-BP} 
	It holds that
	\begin{align*}
		\log \frac{X_{t}^{(i)}(\ell_{T_{\leq t}(v)})}{X_{t}^{(j)}(\ell_{T_{\leq t}(v)})} = \log \frac{X_{1}^{(i)}(\ell_{T_{1}(v)})}{X_{1}^{(j)}(\ell_{T_{1}(v)})} + \sum_{u \in \C^{\rm u}(v)} \log \frac{ 1 + \frac{k\theta}{1-\theta} X_{t-1}^{(i)}(\ell_{T_{\leq t-1}(u)}) }{ 1 + \frac{k\theta}{1-\theta} X_{t-1}^{(j)}(\ell_{T_{\leq t-1}(u)}) }. 
	\end{align*}
\end{lemma}
The above belief propagation formula for $X_t^{(i)}(\ell_{T_{\leq t}(v)})$ is exact. However, it turns out analyzing the density of $X_t^{(i)}(\ell_{T_{\leq t}(v)})$ is hard. Inspired by the ``linearization'' trick for $k=2$, we analyze the following linearized message-passing algorithm. \\

\begin{algorithm}[H] 
	\KwData{A partially labeled tree $T_{\leq t}(\rho)$ with depth $t$ and labels $\ell_{T_{\leq t}(\rho)}$, fixed  $j \in [k]$.} 
	\KwResult{The messages $M^{(i\rightarrow j)}_{t}(\ell_{T_{\leq t}(v)})$, for any $i \in [k]/j$.} 
	initialization: s = 1, and $M_{0}(\ell_{T_{\leq 0}(v)}) = 0, M_{1}^{(i\rightarrow j)}(\ell_{T_{\leq 1}(v)}) = \left(N_{\C^{\rm l}(v)}(i) - N_{\C^{\rm l}(v)}(j) \right) \log \left(1+\frac{k\theta}{1-\theta} \right)$, $\forall v, i\neq j$\; 
	\While{$s \leq t$}{ focus on $(t-s)$-th layer\; \For{$v \in \C_{t-s}(\rho)$ and $v$ unlabeled}{ update messages for the subtree: $M^{(i\rightarrow j)}_{s}(\ell_{T_{\leq s}(v)}) = M^{(i\rightarrow j)}_{1}(\ell_{T_1(v)}) + \theta \cdot \sum\limits_{u \in \C^{\rm u}(v)} M^{(i\rightarrow j)}_{s-1}(\ell_{T_{\leq s-1}(u)})$\; move one layer up: $s = s+1$ \; } } If $\max_{i \in [k]/j} M_{t}^{(i \rightarrow j)}(\ell_{T_{\leq t}(\rho)})>0$, 
	output $\ell(\rho) = \argmax_{i \in [k]/j} M_{t}^{(i \rightarrow j)}(\ell_{T_{\leq t}(\rho)}) $; Else output $\ell(\rho) = j$. 
	\caption{Approximate Message Passing on Partially Labeled $k$-Broadcasting Tree} \label{alg:k-AMP} 
\end{algorithm}
\medskip 
For p-SBM with $k$ blocks, Algorithm~\ref{alg:amp-psbm}, which uses the above Algorithm~\ref{alg:k-AMP} as a sub-routine, will succeed in recovering the labels in the regime above the threshold \eqref{eq:upp.gk}. The theoretical justification is given in the following sections. 

\subsection{Concentration Phenomenon on Messages} 
As in the case $k=2$, here we provide the concentration result on the distribution of approximate messages recursively calculated based on the tree. 
\begin{theorem}[Concentration of Messages for $k$-AMP, $(1-\delta)\theta^2 d>1$] 
	\label{thm:k-AMP} 
	Consider the Approximate Message Passing (AMP) Algorithm~\ref{alg:k-AMP} on the $k$-label broadcasting tree $\text{Tree}_{k}(\theta,d,\delta)$. Assume  $\delta d = O(1)$. With the initial values
	\begin{align*}
		\mu_1  = \theta \delta d \cdot \log \left( 1+ \frac{k\theta}{1-\theta} \right),\quad \sigma_1^2  = \delta d \cdot \log^2 \left( 1+ \frac{k\theta}{1-\theta} \right)
	\end{align*}
	and the factor parameter
	\begin{align*}
		\alpha  : = (1-\delta)\theta^2 d,
	\end{align*}
	 the recursion of the parameters $\mu_t$, $\sigma_t^2$ follows as in Eq.~\eqref{eq: mean-var-update}. 

	For a certain depth $t$, conditionally on $\ell(v) = l$, the moment generating function for $M^{(i \rightarrow j)}_t(\ell_{T_{\leq t}(v)})$ is upper bounded as
	\begin{equation*}
		\Psi_{M^{(i \rightarrow j)}_t} (\lambda) \leq \left\{ 
		\begin{array}{cc}
			e^{\lambda \mu_t} e^{\frac{\lambda^2 \sigma_t^2}{2}}, &~~ i = l \\
			e^{\frac{\lambda^2 \sigma_t^2}{2}}, &~~ i,j \neq l \\
			e^{-\lambda \mu_t} e^{\frac{\lambda^2 \sigma_t^2}{2}}, &~~ j = l 
		\end{array}
		\right. 
	\end{equation*}
	The message-passing Algorithm~\ref{alg:k-AMP} succeeds in recovering the label with probability at least $1 - (k-1) \exp\left( - \frac{\alpha}{2(1+ o(1))}\right)$ when $(1-\delta)\theta^2 d>1$. 
\end{theorem}

	Again, from Theorem~\ref{thm:k-AMP} we can easily get the following recovery proportion guarantee. For the message-passing Algorithm~\ref{alg:AMP}, assume $$\alpha:=(1-\delta)\theta^2 d > 1,$$ and define for any $t$
		\begin{align*}
		\epsilon(t) = \frac{(\alpha-1)^2}{\theta^2\delta d} \frac{1}{\alpha^t-1} + \mathcal{O}(\alpha^{-t}), ~~ \text{with}~~ \lim_{t \rightarrow \infty} \epsilon(t) = 0. 
		\end{align*}
		Then Algorithm~\ref{alg:k-AMP} recovers the label of the root node with probability at least $$ 1 - (k-1) \exp\left( - \frac{(1-\delta)\theta^2 d - 1}{2(1 + \epsilon(t))} \right)$$
		with time complexity $$ \mathcal{O} \left( (k-1) \frac{(\delta d +1)[(1-\delta)d]^t - d}{(1-\delta)d - 1} \right).$$

\subsection{Multiple Testing Lower Bound on Local Algorithm Class}
We conclude the theoretical study with a lower bound for local algorithms for $k$-label broadcasting trees. We bound the distributions of leaf labels, indexed by different root colors and show that in total variation distance, the distributions are indistinguishable (below the threshold in equation~\eqref{eq:low.gk}) from each other as $\delta$ vanishes. 
\begin{theorem}[Limitation for Local Algorithms]
	\label{thm:low.g-k}
	Consider the following measures of revealed labels defined on trees indexed by the root's label: $\mu_{\ell_{T_{\leq t}(\rho)}}^{(i)}, i\in [k]$. Assume $\delta d >1 $, 
	$
	(1-\delta)\theta^2d<1/4
	$ 
	and $$2\delta d\log \left( 1 + \theta^2\left( \frac{1}{\theta+\frac{1-\theta}{k}} + \frac{1}{\frac{1-\theta}{k}} \right) \right) <[1 - 4 (1-\delta)\theta^2d]^2. $$
	Then for any $t>0$, the following bound on the $\chi^2$ distance holds:
	\begin{align*}
		\max_{i,j\in [k]}  \log\left( 1 + d_{\chi^2} \left(  \mu^{(i)}_{\ell_{T_{\leq t}(\rho)}} , \mu^{(j)}_{\ell_{T_{\leq t}(\rho)}} \right) \right) \leq  \frac{2\delta d\log \left( 1 + \theta^2\left( \frac{1}{\theta+\frac{1-\theta}{k}} + \frac{1}{\frac{1-\theta}{k}} \right) \right) }{ 1 - 4 (1-\delta)\theta^2d } \leq k \cdot \frac{2\delta \theta^2 d}{ 1 - 4 (1-\delta)\theta^2d } \left( \frac{1}{1-\theta} + \frac{1}{k\theta + 1-\theta}\right).
	\end{align*}
	Furthermore, it holds that
	$$
	\inf_{\Phi} \sup_{l(\rho) \in [k]} \mbb{P} \left(\Phi(\ell_{T_{\leq t}(\rho)}) \neq \ell(\rho) \right) \geq \frac{1}{2} \left(  1 -   \frac{2\delta \theta^2 d}{ 1 - 4 (1-\delta)\theta^2d } \left( \frac{1}{1-\theta} + \frac{1}{k\theta + 1-\theta}\right) - \frac{1}{k}  \right),
	$$
	where~ $\Phi: \ell_{T\leq t}(\rho) \rightarrow [k]$ is any local estimator mapping from the revealed labels to a decision.
\end{theorem}
	The proof is based on a multiple testing argument in Le Cam's minimax lower bound theory. We would like to remark that condition $4 \cdot (1-\delta)\theta^2 d <1$ can be relaxed to 
	$$
	(1-\delta)\theta^2 d \cdot \left( 1+ 3(1-\theta)(1-\frac{2}{k})\right) < 1.
	$$

\section{Numerical Studies}
\label{sec:num.study}
In this section we apply our approximate message-passing Algorithm~\ref{alg:amp-psbm} to the political blog dataset \citep{adamic2005political}, with a total of 1222 nodes. In the literature, the state-of-the-art result for a global algorithm appears in \citep{jin2015fast}, where the mis-classification rate is $58/1222 = 4.75\%$. Here we run our message-passing Algorithm~\ref{alg:amp-psbm} with three different settings $\delta = 0.1, 0.05, 0.025$, replicating each experiment $50$ times (we sample the revealed nodes independently in $50$ experiments for each $\delta$ specification). As a benchmark, we compare our results to the spectral algorithm on the $(1-\delta)n$ sub-network.  For our message-passing algorithm, we look at the local tree with depth 1 to 5.  The results are summarized as boxplots in Figure~\ref{fig:polblg}. The left figure illustrates the comparison of AMP with depth 1 to 5 and the spectral algorithm, with red, green, blue  boxes corresponding to $\delta = 0.025, 0.05, 0.1$, respectively. The right figure zooms in on the left plot with only AMP depth 2 to 4 and spectral, to better emphasize the difference. 
Remark that if we only look at depth 1, some of the nodes may have no revealed neighbors. In this setting, we classify this node as wrong (this explains why depth-1 error can be larger than 1/2).
\begin{figure}[H] 
	\begin{subfigure}
		{0.48\textwidth} \centering 
		\includegraphics[width=3.3in]{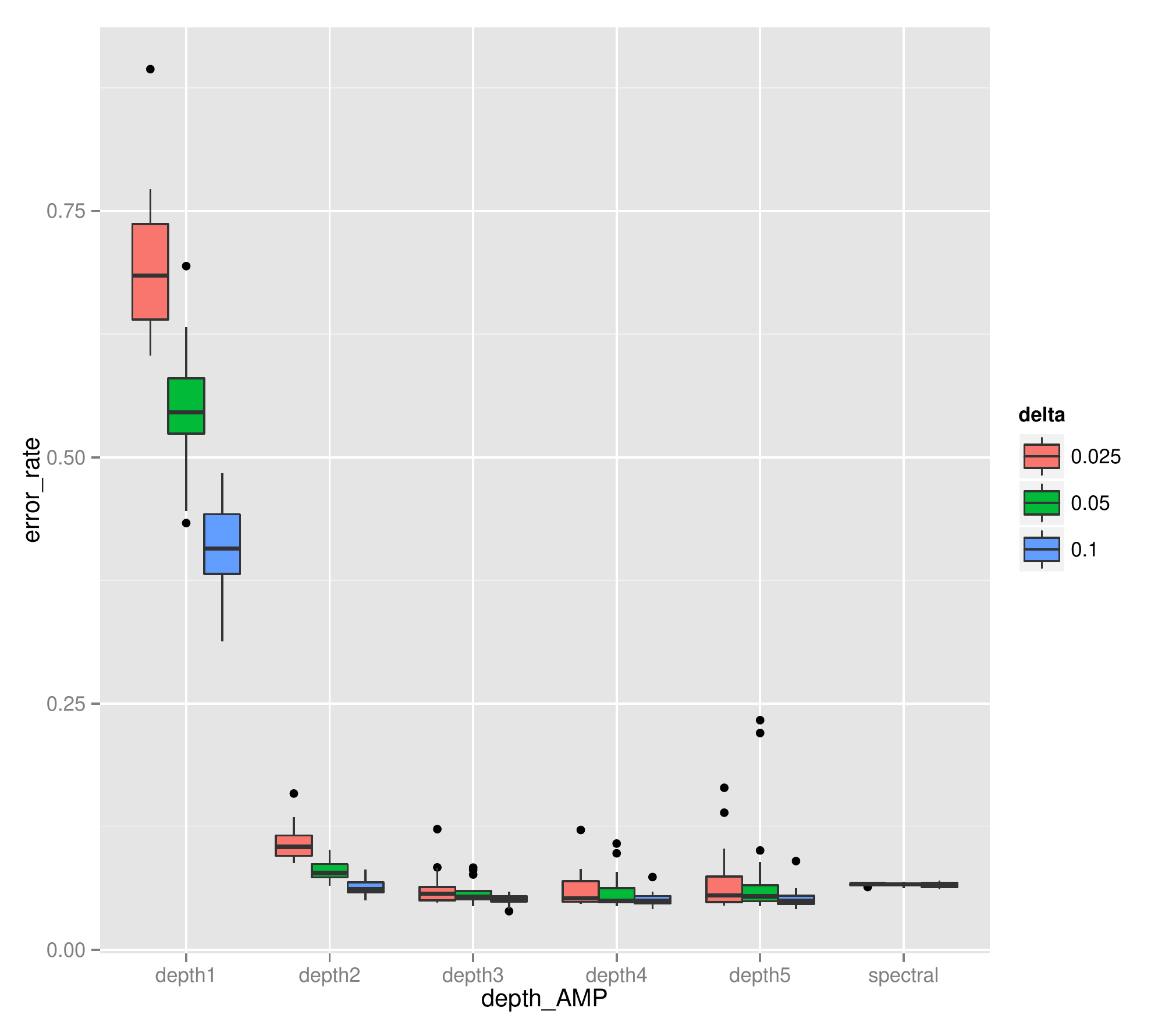} 
	\end{subfigure}
	\begin{subfigure}
		{0.48\textwidth} \centering 
		\includegraphics[width=3.3in]{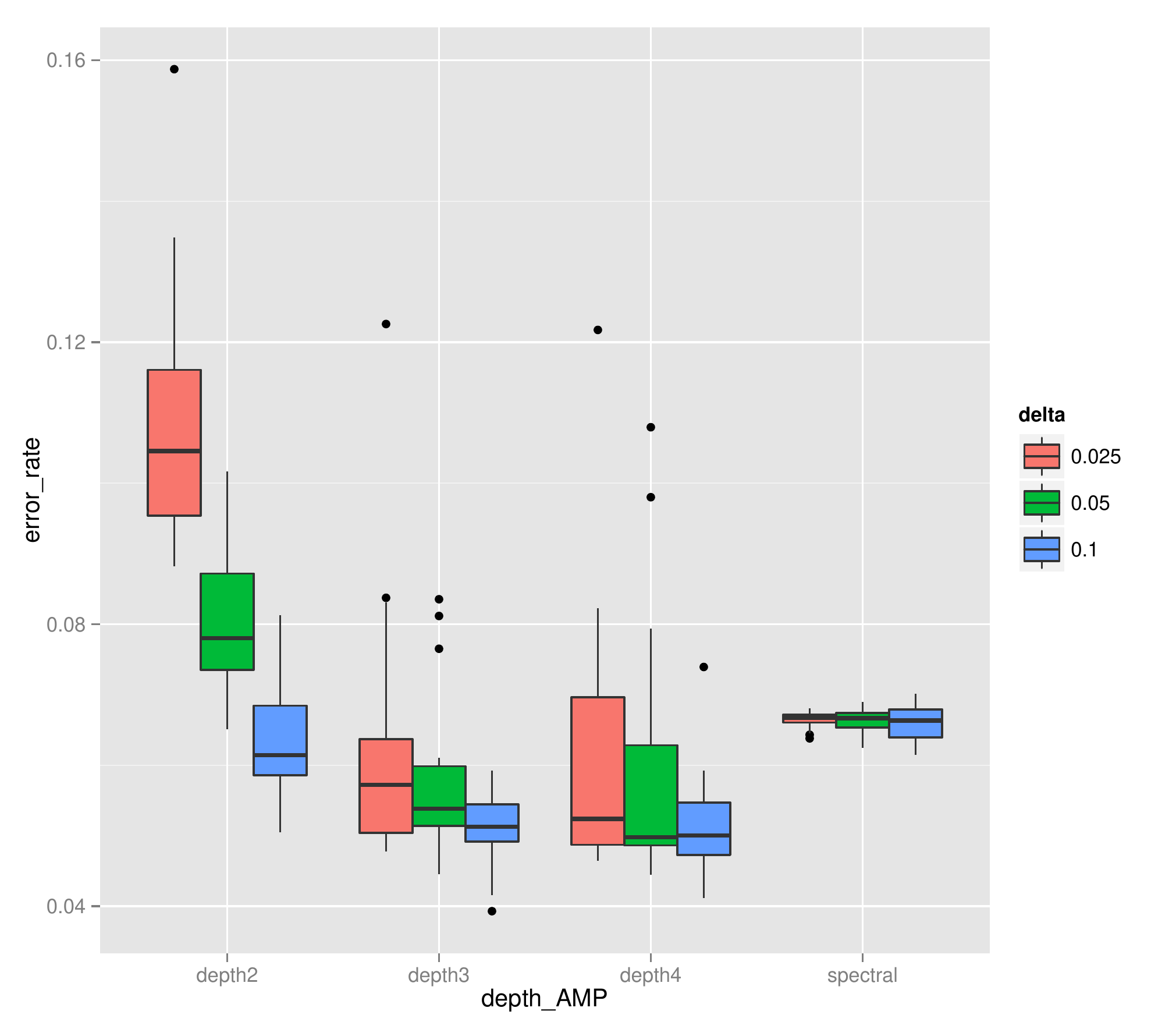} 
	\end{subfigure}
	\caption{AMP algorithm on Political Blog Dataset.} 
	\label{fig:polblg}
\end{figure}
We present in this paragraph some of the statistics of the experiments, extracted from the above Figure~\ref{fig:polblg}. In the case $\delta = 0.1$,  from depth 2-4, the AMP algorithm produces the mis-classification error rate (we took the median over the experiments for robustness) of $6.31\%, 5.22\%, 5.01\%$, while the spectral algorithm produces the error rate $6.68\%$. When $\delta = 0.05$, i.e. about 60 node labels revealed, the error rates are $7.71\%, 5.44\%, 5.08\%$ for the AMP algorithm with depth 2 to 4, contrasted to the spectral algorithm error $6.66\%$. In a more extreme case $\delta = 0.025$ when there are only $~30$ node labels revealed, AMP depth 2-4 has error $10.20\%, 5.71\%, 5.66\%$, while spectral is $6.63\%$. In general, the AMP algorithm with depth 3-4 uniformly beats the vanilla spectral algorithm. Note our AMP algorithm is a distributed decentralized algorithm that can be run in parallel. We acknowledge that the error $\sim 5\%$ (when $\delta$ is very small) is still slightly worse than the state-of-the-art degree-corrected SCORE algorithm in \citep{jin2015fast}, which is $4.75\%$. 


\section{Technical Proofs}
\label{sec:proof}

We will start with two useful Lemmas. Lemma~\ref{lma:coupling.tree} couples the local behavior of a stochastic block model with that of a Galton-Watson branching process. Lemma~\ref{lma:hoeff} is the well-known Hoeffding's inequality. 
\begin{lemma}[Proposition 4.2 in \citep{mossel2012stochastic}]
	\label{lma:coupling.tree}
	Take $t = \bar{t}_{n,k,p,q} \precsim \frac{\log n}{\log [kn(q + \frac{p-q}{k})]} $. There exists a coupling between $(G, \sigma)$ and $(T, \ell)$ such that $(G_{\leq t}, \sigma_{G_{\leq t}}) = (T_{\leq t}, \ell_{T_{\leq t}})$ asymptotically almost surely. Here $(T,\ell)$ corresponds to the broadcast process on a Galton-Watson tree process $T$ with offspring distribution ${\sf Poisson}\left(n(q + \frac{p-q}{k})\right)$, and $(G,\sigma)$ corresponds to the SBM and its labels. 
\end{lemma} 
\bigskip

\begin{lemma}[Hoeffding's Inequality]
	\label{lma:hoeff}
	Let X be any real-valued random variable with expected value $\mathbb{E}X = 0$ and such that $a \leq X \leq b$ almost surely. Then, for all $\lambda>0$,
	$$
	\mathbb{E} \left[ e^{\lambda X} \right] \leq \exp \left( \frac{\lambda^2 (b - a)^2}{8} \right).
	$$
\end{lemma}
\bigskip

Let us now derive the algorithms for belief propagation. 
\begin{proof}[Derivation of Belief Propagation, $k=2$]
\label{pf:deri.k=2}
The algorithm we rely on is Belief Propagation (or MAP). We recursively calculate the posterior probability $\mbb{P} \left( \ell(\rho) = + | \ell_{T_{\leq t}(\rho)} \right) $ backwards from the leaf of the tree. The recursion holds from $t\rightarrow t-1$ via Bayes Theorem 
\begin{align*}
	& \quad \mbb{P}\left( \ell(\rho) = + | \ell_{T_{\leq t}(\rho)} \right) \\
	& = \frac{\mbb{P} \left( \ell_{T_{\leq t}(\rho)} | \ell(\rho) = + \right) \mbb{P} \left(\ell(\rho) = +\right) }{ \mbb{P} \left( \ell_{T_{\leq t}(\rho)} | \ell(\rho) = + \right) \mbb{P} \left(\ell(\rho) = +\right) + \mbb{P} \left( \ell_{T_{\leq t}(\rho)} | \ell(\rho) = - \right) \mbb{P} \left(\ell(\rho) = -\right)} \\
	& = \frac{\prod\limits_{v \in \C^{\rm l}(\rho)}\mbb{P} \left( \ell(v)| \ell(\rho) = + \right) \prod\limits_{v \in \C^{\rm u}(\rho)} \mbb{P} \left( \ell_{T_{\leq t-1}}(v) | \ell(\rho) = + \right) }{\prod\limits_{v \in \C^{\rm l}(\rho)}\mbb{P} \left( \ell(v)| \ell(\rho) = + \right) \prod\limits_{v \in \C^{\rm u}(\rho)} \mbb{P} \left( \ell_{T_{\leq t-1}}(v) | \ell(\rho) = + \right) + \prod\limits_{v \in \C^{\rm l}(\rho)}\mbb{P} \left( \ell(v)| \ell(\rho) = - \right) \prod\limits_{v \in \C^{\rm u}(\rho)} \mbb{P} \left( \ell_{T_{\leq t-1}}(v) | \ell(\rho) = - \right)}. 
\end{align*}
Here we denote for the $\delta d$ labeled nodes 
\begin{align*}
	M_1(\ell_{T_{1}(\rho)}):= \log \frac{\mbb{P}\left( \ell(\rho) = + | \ell_{T_{\leq 1}(\rho)}\right)}{\mbb{P}\left( \ell(\rho) = - | \ell_{T_{\leq 1}(\rho)}\right)} = \log \frac{\prod\limits_{v \in \C^{\rm l}(\rho)}\mbb{P} \left( \ell(v)| \ell(\rho) = + \right) }{ \prod\limits_{v \in \C^{\rm l}(\rho)}\mbb{P} \left( \ell(v)| \ell(\rho) = - \right) } = \left(N_{\C^{\rm l}(\rho)}(+) - N_{\C^{\rm l}(\rho)}(-) \right) \log \frac{1+\theta}{1-\theta}, 
\end{align*}
where $N_{\C^{\rm l}(\rho)}(+) \sim {\rm Binom}\left(\delta d, \frac{1+\theta}{2}\right)$, $N_{\C^{\rm l}(\rho)}(-)= \delta d - N_{\C^{\rm l}(\rho)}(+)$. For the unlabeled nodes $v \in \C^{\rm u}(\rho)$ define 
\begin{align*}
	& M_{t-1}(\ell_{T_{\leq t-1}(v)}):= \log \frac{ \mbb{P} \left( \ell(v) = + | \ell_{T_{\leq t-1}}(v) \right)}{ \mbb{P} \left(\ell(v) = - | \ell_{T_{\leq t-1}}(v) \right)} = \log \frac{ 1 + \frac{\mbb{P} \left( \ell_{T_{\leq t-1}}(v)| \ell(v) = + \right) - \mbb{P} \left( \ell_{T_{\leq t-1}}(v)| \ell(v) = - \right)}{\mbb{P} \left( \ell_{T_{\leq t-1}}(v)| \ell(v) = + \right) + \mbb{P} \left( \ell_{T_{\leq t-1}}(v)| \ell(v) = - \right)} }{ 1 - \frac{\mbb{P} \left( \ell_{T_{\leq t-1}}(v)| \ell(v) = + \right) - \mbb{P} \left( \ell_{T_{\leq t-1}}(v)| \ell(v) = - \right)}{\mbb{P} \left( \ell_{T_{\leq t-1}}(v)| \ell(v) = + \right) + \mbb{P} \left( \ell_{T_{\leq t-1}}(v)| \ell(v) = - \right)} }, 
\end{align*}
which means 
\begin{align*}
	\frac{\mbb{P} \left( \ell_{T_{\leq t-1}}(v)| \ell(v) = + \right) - \mbb{P} \left( \ell_{T_{\leq t-1}}(v)| \ell(v) = - \right)}{\mbb{P} \left( \ell_{T_{\leq t-1}}(v)| \ell(v) = + \right) + \mbb{P} \left( \ell_{T_{\leq t-1}}(v)| \ell(v) = - \right)} = \tanh \left( M_{t-1}(\ell_{T_{\leq t-1}(v)})/2 \right). 
\end{align*}
Now we have 
\begin{align*}
	&\quad \log \frac{ \mbb{P} \left( \ell_{T_{\leq t-1}}(v) | \ell(\rho) = + \right)}{ \mbb{P} \left( \ell_{T_{\leq t-1}}(v) | \ell(\rho) = - \right)} \\
	& = \log \frac{ \mbb{P} \left( \ell_{T_{\leq t-1}}(v), \ell(v) = + | \ell(\rho) = + \right) + \mbb{P} \left( \ell_{T_{\leq t-1}}(v), \ell(v) = - | \ell(\rho) = + \right)}{ \mbb{P} \left( \ell_{T_{\leq t-1}}(v),\ell(v) = + | \ell(\rho) = - \right)+ \mbb{P} \left( \ell_{T_{\leq t-1}}(v), \ell(v) = - | \ell(\rho) = - \right)} \\
	& = \log \frac{ \mbb{P} \left( \ell_{T_{\leq t-1}}(v)| \ell(v) = + \right) \cdot \frac{1+\theta}{2} + \mbb{P} \left( \ell_{T_{\leq t-1}}(v)| \ell(v) = - \right) \cdot \frac{1-\theta}{2} }{ \mbb{P} \left( \ell_{T_{\leq t-1}}(v) | \ell(v) = + \right) \cdot \frac{1-\theta}{2}+ \mbb{P} \left( \ell_{T_{\leq t-1}}(v) | \ell(v) = - \right) \cdot \frac{1+\theta}{2}} \\
	& = \log \frac{ 1 + \theta \cdot \frac{\mbb{P} \left( \ell_{T_{\leq t-1}}(v)| \ell(v) = + \right) - \mbb{P} \left( \ell_{T_{\leq t-1}}(v)| \ell(v) = - \right)}{\mbb{P} \left( \ell_{T_{\leq t-1}}(v)| \ell(v) = + \right) + \mbb{P} \left( \ell_{T_{\leq t-1}}(v)| \ell(v) = - \right)} }{1 - \theta \cdot \frac{\mbb{P} \left( \ell_{T_{\leq t-1}}(v)| \ell(v) = + \right) - \mbb{P} \left( \ell_{T_{\leq t-1}}(v)| \ell(v) = - \right)}{\mbb{P} \left( \ell_{T_{\leq t-1}}(v)| \ell(v) = + \right) + \mbb{P} \left( \ell_{T_{\leq t-1}}(v)| \ell(v) = - \right)} } \\
	& = \log \frac{ 1 + \theta \cdot \tanh \left( M_{t-1}(\ell_{T_{\leq t-1}(v)}) / 2 \right) }{1 - \theta \cdot \tanh \left( M_{t-1}(\ell_{T_{\leq t-1}(v)}) / 2 \right) } \\
	& = f_{\theta}\left(M_{t-1}(\ell_{T_{\leq t-1}(v)})\right). 
\end{align*}
Thus we have the following recursion 
\begin{align*}
	M_{t} (\ell_{T_{\leq t}(\rho)}) = M_1(\ell_{T_{1}(\rho)}) + \sum_{v \in \C^{\rm u}(\rho)} f_{\theta}\left(M_{t-1}(\ell_{T_{\leq t-1}(v)})\right) 
\end{align*}
with $|\C^{\rm u}(\rho)| = (1-\delta)d$. The notation $M_{t}(\ell_{T_{\leq t}(\rho)})$ can be viewed as the messages (logit) of the root label $\ell(\rho)$ for the depth $t$ tree that rooted from $\rho$. $M_1(\ell_{T_{1}(\rho)})$ denotes the message on the labels on depth $1$ layer with root $\rho$. These messages denote the belief of the labeling of the node $\rho$ based on the random labels $\ell_{T_{\leq t}(\rho)}$. 
\end{proof}
\bigskip

\begin{proof}[Derivation of Lemma~\ref{lma:k-BP}, BP, general $k$] 
	
	Note $X_{t-1}^{(i)}(\ell_{T_{\leq t-1}(u)})$ are $(1-\delta)d$ i.i.d conditionally on $\ell(v)$. The initial message is
	\begin{align*}
	\log \frac{X_{1}^{(i)}(\ell_{T_{1}(v)})}{X_{1}^{(j)}(\ell_{T_{1}(v)})} = \left(N_{\C^{\rm l}(v)}(i) - N_{\C^{\rm l}(v)}(j) \right) \cdot \log \left(1+ \theta \right). 
	\end{align*}
	For the case when $\ell(v) = l$, we have 
	\begin{align*}
	N_{\C^{\rm l}(v)}(l) \sim {\rm Binom}(\delta d, \theta + \frac{1-\theta}{k}) \\
	N_{\C^{\rm l}(v)}(i) \sim {\rm Binom}(\delta d, \frac{1-\theta}{k}), i\in [k]/l. 
	\end{align*}
	Then
	\begin{align*}
	\log \frac{X_{t}^{(i)}(\ell_{T_{\leq t}(\rho)})}{X_{t}^{(j)}(\ell_{T_{\leq t}(\rho)})} & = \log \frac{ \mbb{P}\left( \ell_{T_{\leq t}(\rho)} | \ell(\rho) = i \right) }{ \mbb{P}\left( \ell_{T_{\leq t}(\rho)} | \ell(\rho) = j \right) } \\
	& = \log \frac{\prod\limits_{v \in \C^{\rm l}(\rho)}\mbb{P} \left( \ell(v)| \ell(\rho) = i \right) \prod\limits_{v \in \C^{\rm u}(\rho)} \mbb{P} \left( \ell_{T_{\leq t-1}}(v) | \ell(\rho) = i \right)}{\prod\limits_{v \in \C^{\rm l}(\rho)}\mbb{P} \left( \ell(v)| \ell(\rho) = j \right) \prod\limits_{v \in \C^{\rm u}(\rho)} \mbb{P} \left( \ell_{T_{\leq t-1}}(v) | \ell(\rho) = j \right)} \\
	& = \log \frac{X_{1}^{(i)}(\ell_{T_{1}(v)})}{X_{1}^{(j)}(\ell_{T_{1}(v)})} + \sum_{v \in \C^{\rm u}(\rho)} \log \frac{\theta \cdot \mbb{P} \left( \ell_{T_{\leq t-1}}(v) | \ell(v) = i \right) + \frac{1-\theta}{k} \cdot \sum_{l \in [k]} \mbb{P} \left( \ell_{T_{\leq t-1}}(v) | \ell(v) = l \right) }{\theta \cdot \mbb{P} \left( \ell_{T_{\leq t-1}}(v) | \ell(v) = j \right) + \frac{1-\theta}{k} \cdot \sum_{l \in [k]} \mbb{P} \left( \ell_{T_{\leq t-1}}(v) | \ell(v) = l \right)}\\
	& = \log \frac{X_{1}^{(i)}(\ell_{T_{1}(v)})}{X_{1}^{(j)}(\ell_{T_{1}(v)})} + \sum_{v \in \C^{\rm u}(\rho)} \log \frac{ 1 + \theta X_{t-1}^{(i)}(\ell_{T_{\leq t-1}(v)}) }{ 1 + \theta X_{t-1}^{(j)}(\ell_{T_{\leq t-1}(v)}) } .
	\end{align*}
\end{proof}

\bigskip

Now we are ready to prove the main theoretical results. First, we focus on the $k=2$ case and prove the broadcasting tree version of Theorem~\ref{thm:AMP.k=2} and Theorem~\ref{thm:low.k=2}, under the assumption the tree is regular. Later, based on these two theorems, Theorem~\ref{thm:sbm.k=2} for p-SBM ($k=2$) is proved.

\begin{proof}[Proof of Theorem~\ref{thm:AMP.k=2}]
	We focus on a regular tree where each node has $(1-\delta)d$ unlabeled children and $\delta d$ labeled children.
	For $t=1$, the results follow from Hoeffding's inequality directly because $$ M_1(\ell_{T_{1}(\rho)}) = \left(N_{\C^{\rm l}(\rho)}(+) - N_{\C^{\rm l}(\rho)}(-) \right) \log \frac{1+\theta}{1-\theta} . $$ Let us use induction to prove the remaining claim. Assume for tree with depth $t-1$ rooted from $u$, 
	for any $\lambda > 0$ 
	\begin{align*}
		\mathbb{E}\left[ e^{\lambda M_{t-1}(\ell_{T_{\leq t-1}(u)})} | \ell(u) = + \right] \leq e^{\lambda \mu_{t-1}} \cdot e^{\frac{\lambda^2}{2} \sigma_{t-1}^2 }, \\
		\mathbb{E}\left[ e^{\lambda M_{t-1}(\ell_{T_{\leq t-1}(u)})} | \ell(u) = - \right] \leq e^{-\lambda \mu_{t-1}} \cdot e^{\frac{\lambda^2}{2} \sigma_{t-1}^2 }. 
	\end{align*}
	These will further imply, conditionally on $\ell(u) = +$, 
		\begin{align*}
			M_{t-1}(\ell_{T_{\leq t-1}(u)}) \in \mu_{t-1} \pm x \cdot \sigma_{t-1}; 
		\end{align*}
		and conditionally on $\ell(v) = -$, 
		\begin{align*}
			M_{t-1}(\ell_{T_{\leq t-1}(u)}) \in - \mu_{t-1} \pm x \cdot \sigma_{t-1}; 
		\end{align*}
		both with probability at least $1-2\exp(x^2/2)$.
		Now, recall the recursion for AMP: 
	\begin{align*}
		M_{t} (\ell_{T_{\leq t}(v)}) = M_1(\ell_{T_{1}(v)}) + \theta \cdot \sum_{u \in \C^{\rm u}(v)} M_{t-1}(\ell_{T_{\leq t-1}(u)}).
	\end{align*}
	For the moment generating function we have 
	\begin{align*}
		&\quad \mathbb{E}\left[ e^{\lambda M_{t}(\ell_{T_{\leq t}(v)})} | \ell(v) = + \right] \\
		& \leq e^{\lambda \left(\theta \delta d \cdot \log\frac{1+\theta}{1-\theta}\right)} e^{\frac{\lambda^2}{2} \left( \sqrt{\delta d} \cdot \log \frac{1+\theta}{1-\theta} \right)^2} \cdot \prod_{u \in \C^{\rm u}(v)}\mathbb{E}\left[ e^{\lambda \theta M_{t-1}(\ell_{T_{\leq t-1}(u)})} | \ell(v) = + \right] \\
		& = e^{\lambda \mu_{1}} e^{\frac{\lambda^2}{2} \sigma_{1}^2 } \cdot \prod_{u \in \C^{\rm u}(v)}\mathbb{E}\left[ e^{\lambda \theta M_{t-1}(\ell_{T_{\leq t-1}(u)})} | \ell(v) = + \right]. 
	\end{align*}
	The last term in the previous equation can be written as
	\begin{align}
		& \quad \mathbb{E}\left[ e^{\lambda \theta M_{t-1}(\ell_{T_{\leq t-1}(u)})} | \ell(v) = + \right] \nonumber \\
		& \leq e^{\frac{\lambda^2 \theta^2}{2} \sigma_{t-1}^2 } \cdot \left\{ \frac{1+\theta}{2} e^{\lambda \theta \mu_{t-1}} + \frac{1-\theta}{2}e^{-\lambda \theta \mu_{t-1}} \right\} \label{eq:hoeff1}\\
		& \leq e^{\frac{\lambda^2 \theta^2}{2} \sigma_{t-1}^2 } \cdot e^{\lambda \theta \left(\frac{1+\theta}{2} \mu_{t-1} - \frac{1-\theta}{2} \mu_{t-1} \right) } \cdot e^{\frac{\lambda^2 \theta^2}{2} \mu_{t-1}^2} \label{eq:hoeff2}\\
		& = e^{\frac{\lambda^2 \theta^2}{2} \sigma_{t-1}^2 } \cdot e^{\lambda \theta^2 \mu_{t-1} } \cdot e^{\frac{\lambda^2 \theta^2}{2} \mu_{t-1}^2} \nonumber 
	\end{align}
	where equation \eqref{eq:hoeff1} to \eqref{eq:hoeff2} relies on Hoeffding's lemma: for a random variable $Y = \theta \mu_{t-1}$ with probability $\frac{1+\theta}{2}$ and $Y = -\theta \mu_{t-1}$ with probability $\frac{1-\theta}{2}$, $$\Psi_Y(\lambda)= \mbb{E} e^{\lambda Y} \leq e^{\lambda \mbb{E}Y} e^{\frac{\lambda^2}{2} \theta^2 \mu_{t-1}^2} = e^{\lambda \left( \frac{1+\theta}{2} \theta \mu_{t-1} - \frac{1- \theta}{2} \theta \mu_{t-1} \right)} e^{\frac{\lambda^2}{2} \theta^2 \mu_{t-1}^2}.$$ Thus 
	\begin{align*}
		&\quad \mathbb{E}\left[ e^{\lambda M_{t}(\ell_{T_{\leq t}(v)})} | \ell(v) = + \right] \\
		& \leq e^{\frac{\lambda^2}{2} \sigma_{1}^2 } \cdot e^{\lambda \mu_{1}} \cdot \left\{ e^{\frac{\lambda^2 \theta^2}{2} \sigma_{t-1}^2 } \cdot e^{\lambda \theta^2 \mu_{t-1} } \cdot e^{\frac{\lambda^2 \theta^2}{2} \mu_{t-1}^2} \right\}^{(1-\delta)d} \\
		& = e^{\lambda (\mu_1 + \alpha \mu_{t-1})} \cdot e^{\frac{\lambda^2}{2}(\sigma_1^2 + \alpha \sigma_{t-1}^2 + \alpha \mu_{t-1}^2)} \\
		& = e^{\lambda \mu_t} \cdot e^{\frac{\lambda^2}{2} \sigma_{t}^2}. 
	\end{align*}
	When $\ell(v) = -$, we have 
	\begin{align*}
		&\quad \mathbb{E}\left[ e^{\lambda M_{t}(\ell_{T_{\leq t}(v)})} | \ell(v) = - \right] \\
		&\leq e^{\frac{\lambda^2}{2} \sigma_{1}^2 } \cdot e^{- \lambda \mu_{1}} \cdot \prod_{u \in \C^{\rm u}(v)}\mathbb{E}\left[ e^{\lambda \theta M_{t-1}(\ell_{T_{\leq t-1}(u)})} | \ell(v) = - \right] \\
		&\leq e^{\frac{\lambda^2}{2} \sigma_{1}^2 } \cdot e^{- \lambda \mu_{1}} \cdot \left\{ e^{\frac{\lambda^2 \theta^2}{2} \sigma_{t-1}^2 } \cdot \left\{ \frac{1+\theta}{2} e^{-\lambda \theta \mu_{t-1}} + \frac{1-\theta}{2}e^{\lambda \theta \mu_{t-1}} \right\} \right\}^{(1-\delta)d} \\
		& \leq e^{\frac{\lambda^2}{2} \sigma_{1}^2 } \cdot e^{- \lambda \mu_{1}} \cdot \left\{ e^{\frac{\lambda^2 \theta^2}{2} \sigma_{t-1}^2 } \cdot e^{-\lambda \theta^2 \mu_{t-1} } \cdot e^{\frac{\lambda^2 \theta^2}{2} \mu_{t-1}^2} \right\}^{(1-\delta)d} \\
		& = e^{-\lambda \mu_t} \cdot e^{\frac{\lambda^2}{2} \sigma_{t}^2}. 
	\end{align*}
	This completes the proof.
\end{proof}

\bigskip

\begin{proof}[Proof of Theorem~\ref{thm:low.k=2}]
	\label{pf:low.k=2}
	Define the measure $\mu_{\ell_{T_{\leq t}(\rho)}}^{+}$ on the revealed labels, for a depth $t$ tree rooted from $\rho$ with label $\ell(\rho) = +$ (and similarly define $\mu_{\ell_{T_{\leq t}(\rho)}}^{-}$). We have the following recursion formula
\begin{align*}
	\mu_{\ell_{T_{\leq t}(\rho)}}^{+} = \left( \frac{1+\theta}{2} \right)^{N_{\C^l(\rho)}} \left( \frac{1-\theta}{2} \right)^{\delta d - N_{\C^l(\rho)}} \prod_{v \in \C^u(\rho)} \left[ \frac{1+\theta}{2} \cdot \mu_{\ell_{\leq t-1}(v)}^{+} + \frac{1-\theta}{2} \cdot \mu_{\ell_{\leq t-1}(v)}^{-}  \right]. 
\end{align*}
Recall that the $\chi^2$ distance between two absolute continuous measures $\mu(x),\nu(x)$ is
$$
d_{\chi^2}(\mu,\nu) = \int \frac{\mu^2}{\nu} dx - 1, 
$$
and we have the total variation distance between these two measures is upper bounded by the $\chi^2$ distance
\begin{align*}
	d_{TV} \left( \mu_{\ell_{T_{\leq t}(\rho)}}^{+},\mu_{\ell_{T_{\leq t}(\rho)}}^{-} \right) \leq \sqrt{d_{\chi^2} \left( \mu_{\ell_{T_{\leq t}(\rho)}}^{+},\mu_{\ell_{T_{\leq t}(\rho)}}^{-} \right) }.
\end{align*}
Let us upper bound the symmetric version of $\chi^2$ distance defined as
\begin{align*}
d^t_{\chi^2} := \max \left\{d_{\chi^2} \left( \mu_{\ell_{T_{\leq t}(\rho)}}^{+},\mu_{\ell_{T_{\leq t}(\rho)}}^{-} \right), ~ d_{\chi^2} \left( \mu_{\ell_{T_{\leq t}(\rho)}}^{-},\mu_{\ell_{T_{\leq t}(\rho)}}^{+} \right) \right\} .
\end{align*}
Note that
\begin{align*}
	&d_{\chi^2} \left( \mu_{\ell_{T_{\leq t}(\rho)}}^{+},\mu_{\ell_{T_{\leq t}(\rho)}}^{-} \right) \\
	&= \left( 1 + \frac{4\theta^2}{1-\theta^2} \right)^{\delta d} \cdot \left[ 1 + d_{\chi^2} \left( \frac{1+\theta}{2} \cdot \mu_{\ell_{\leq t-1}(v)}^{+} + \frac{1-\theta}{2} \cdot \mu_{\ell_{\leq t-1}(v)}^{-}, \frac{1+\theta}{2} \cdot \mu_{\ell_{\leq t-1}(v)}^{-} + \frac{1-\theta}{2} \cdot \mu_{\ell_{\leq t-1}(v)}^{+} \right) \right]^{(1-\delta)d} - 1,
\end{align*}
and for the RHS, we have the expression
\begin{align*}
	d_{\chi^2} \left( \frac{1+\theta}{2} \cdot \mu_{\ell_{\leq t-1}(v)}^{+} + \frac{1-\theta}{2} \cdot \mu_{\ell_{\leq t-1}(v)}^{-}, \frac{1+\theta}{2} \cdot \mu_{\ell_{\leq t-1}(v)}^{-} + \frac{1-\theta}{2} \cdot \mu_{\ell_{\leq t-1}(v)}^{+} \right) = \theta^2 \int \frac{(\mu_{\ell_{\leq t-1}(v)}^{+}  - \mu_{\ell_{\leq t-1}(v)}^{-})^2}{\frac{1+\theta}{2} \cdot \mu_{\ell_{\leq t-1}(v)}^{-} + \frac{1-\theta}{2} \cdot \mu_{\ell_{\leq t-1}(v)}^{+} } d x.
\end{align*}
Recalling Jensen's inequality, RHS of the above equation is further upper bounded by
\begin{align*}
	{\rm RHS} & \leq \theta^2 \int (\mu_{\ell_{\leq t-1}(v)}^{+}  - \mu_{\ell_{\leq t-1}(v)}^{-})^2 \cdot \left[ \frac{1+\theta}{2} \cdot \frac{1}{\mu_{\ell_{\leq t-1}(v)}^{-}} + \frac{1-\theta}{2} \cdot \frac{1}{\mu_{\ell_{\leq t-1}(v)}^{+}}  \right] dx \\
	& = \theta^2  \left[ \frac{1+\theta}{2} d_{\chi^2}\left(\mu_{\ell_{T_{\leq t}(\rho)}}^{+},\mu_{\ell_{T_{\leq t}(\rho)}}^{-}\right) + \frac{1-\theta}{2} d_{\chi^2}\left(\mu_{\ell_{\leq t-1}(v)}^{-},\mu_{\ell_{\leq t-1}(v)}^{+}\right)  \right].
\end{align*}
Thus 
\begin{align*}
  &\max \left\{	d_{\chi^2} \left( \frac{1+\theta}{2} \cdot \mu_{\ell_{\leq t-1}(v)}^{+} + \frac{1-\theta}{2} \cdot \mu_{\ell_{\leq t-1}(v)}^{-}, \frac{1+\theta}{2} \cdot \mu_{\ell_{\leq t-1}(v)}^{-} + \frac{1-\theta}{2} \cdot \mu_{\ell_{\leq t-1}(v)}^{+} \right),  \right.\\
  &\quad \quad \quad  \left.   d_{\chi^2} \left( \frac{1+\theta}{2} \cdot \mu_{\ell_{\leq t-1}(v)}^{-} + \frac{1-\theta}{2} \cdot \mu_{\ell_{\leq t-1}(v)}^{+}, \frac{1+\theta}{2} \cdot \mu_{\ell_{\leq t-1}(v)}^{+} + \frac{1-\theta}{2} \cdot \mu_{\ell_{\leq t-1}(v)}^{-} \right) \right\} \\
  \leq & \theta^2 \max\left\{ d_{\chi^2}\left(\mu_{\ell_{\leq t-1}(v)}^{+},\mu_{\ell_{\leq t-1}(v)}^{-}\right) , d_{\chi^2}\left(\mu_{\ell_{\leq t-1}(v)}^{-},\mu_{\ell_{\leq t-1}(v)}^{+}\right)  \right\} =\theta^2 d_{\chi^2}^{t-1} .
\end{align*}
Therefore, we have
\begin{align*}
	\log\left(1+
	d_{\chi^2}^t \right) \leq \delta d \cdot \log \left( 1 + \frac{4\theta^2}{1 - \theta^2} \right) + (1-\delta) d  \cdot \log \left(1 + \theta^2 \cdot d_{\chi^2}^{t-1} \right).
\end{align*}	
If $(1-\delta) \theta^2 d<1$, denote the fixed point of the above equation as $c^*$ (the existence is manifested by the following bound ~\eqref{eq:fix-point}), i.e.,
$$
\log (1+ c^*) = \delta d \cdot \log \left( 1 + \frac{4\theta^2}{1 - \theta^2} \right) + (1-\delta)d \cdot \log (1+\theta^2 \cdot c^*).
$$
Due to the fact that $x - \frac{1}{2}x^2 \leq \log(1+x)\leq x$, we have the following upper bound
\begin{align}
&c^* - \frac{1}{2} (c^*)^2 \leq \log (1+ c^*) = \delta d \cdot \log \left( 1 + \frac{4\theta^2}{1 - \theta^2} \right) + (1-\delta)d \cdot \log (1+\theta^2 c^*) \leq \delta d \cdot \log \left( 1 + \frac{4\theta^2}{1 - \theta^2} \right) + (1-\delta)\theta^2 d \cdot c^* \\
&\quad  \quad  \text{and thus}~~~c^* \leq  \frac{\delta d \cdot \log \left( 1 + \frac{4\theta^2}{1 - \theta^2} \right)}{1 - (1-\delta)\theta^2 d}  \cdot \frac{2}{1+\sqrt{1 - 2\frac{\delta d \cdot \log \left( 1 + \frac{4\theta^2}{1 - \theta^2} \right)}{(1 - (1-\delta)\theta^2 d)^2}}} \label{eq:fix-point}.
\end{align}
If we have
$$
d^{t-1}_{\chi^2} \leq c^*
$$
it is easy to see that 
\begin{align*}
	\log\left(1+
	d_{\chi^2}^t \right) &\leq \delta d \cdot \log \left( 1 + \frac{4\theta^2}{1 - \theta^2} \right) + (1-\delta) d  \cdot \log \left(1 + \theta^2 \cdot d_{\chi^2}^{t-1} \right) \\
	&\leq \delta d \cdot \log \left( 1 + \frac{4\theta^2}{1 - \theta^2} \right) + (1-\delta)d \cdot \log (1+\theta^2 \cdot c^*) = \log (1+c^*),
\end{align*}
which implies $d^{t}_{\chi^2} \leq c^*$. Therefore we only need to verify $d^1_{\chi^2} \leq c^*$, which is trivial. 
Thus we have the bound,
\begin{align*}
	\limsup_{t \rightarrow \infty}~ 
	d_{\chi^2}^t  \leq c^* \leq 2\frac{\delta d \cdot \log \left( 1 + \frac{4\theta^2}{1 - \theta^2} \right)}{1 - (1-\delta)\theta^2 d},
\end{align*}
provided $\frac{2\delta d \cdot \log \left( 1 + \frac{4\theta^2}{1 - \theta^2} \right)}{(1-(1-\delta)\theta^2 d)^2} < 1 $.
So far we have proved
\begin{align*}
	\limsup_{t\rightarrow \infty}~ d_{TV}^{t} \leq  \limsup_{t\rightarrow \infty} ~ \left( d_{\chi^2}^{t}\right)^{1/2} \leq \left\{ \frac{2\delta d \cdot \log \left( 1 + \frac{4\theta^2}{1 - \theta^2} \right)}{1 - (1-\delta)\theta^2 d}  \right\}^{1/2}.
\end{align*}
Through Le Cam's Lemma,  the error rate, for all local algorithms, is at least 
$$
\inf_{\Phi} \sup_{l \in \{+,-\}} \mbb{P}_l (\Phi \neq l) \geq \frac{1 - \left\{ \frac{2\delta d \cdot \log \left( 1 + \frac{4\theta^2}{1 - \theta^2} \right)}{1 - (1-\delta)\theta^2 d}  \right\}^{1/2} }{2}.
$$

\end{proof}

\bigskip

Now we are ready to prove Theorem~\ref{thm:sbm.k=2} with the help of Lemma~\ref{lma:coupling.tree}. The main task in the proof of Theorem~\ref{thm:sbm.k=2} is to extend Theorems~\ref{thm:AMP.k=2} and ~\ref{thm:low.k=2} from the regular tree case to the general branching tree case with the matching branching number. Since the general branching tree is a random tree with a varied structure, we need to prove that versions of upper and lower bounds from the earlier proofs hold almost surely for this random tree. The proof requires new ideas employing different notions of the ``branching number'' \citep{lyons2005probability}. 

\begin{proof}[Proof of Theorem~\ref{thm:sbm.k=2}]
	For the  regular tree, the upper and lower bounds have been already proved in Theorem~\ref{thm:AMP.k=2} and Theorem~\ref{thm:low.k=2}. Instead of a $(1-\delta)d$ regular tree, we need to prove the theorem for Galton-Watson tree with offspring distribution ${\sf Poisson}((1-\delta)d)$ (recall the a.a.s. coupling between local tree of SBM and Galton-Watson tree from Lemma~\ref{lma:coupling.tree}). Note these two trees share the same branching number ${\rm br}(T) = (1-\delta)d$ almost surely.  We will use the following  equivalent definitions of the branching number \citep{lyons2005probability}  for a tree $T$ rooted at $\rho$:
	\begin{itemize}
		\item[Flow]: a \textit{flow} is non-negative function on the edges of $T$, with the property that for each non-root vertex $x$, ${\sf flow}((z,x)) = \sum_{i=1}^d {\sf flow}((x,y_i))$ if $x$ has parent $z$ and children $y_1,\ldots, y_d$. We say that ${\sf flow}(e)$ is the amount of water flowing along edge $e$ with the total amount of water being ${\rm flow}(T) = \sum_{v \in \C(\rho)} {\sf flow}((\rho, v))$. Consider the following restriction on a flow: given $\lambda\geq 1$, ${\sf flow}((x,z)) \leq \lambda^{-n}$ for an edge $(x,z)$ with distance $n$ from $\rho$. \textit{Branching number ${\sf br}(T)$} of a tree $T$ is the supremum over $\lambda$ that admits a positive total amount of water ${\rm flow}_{\lambda}(T)>0$ to flow through $T$. Denote for a node $v$ with parent $u$, the ${\sf flow}(v):= {\sf flow}((u,v))$.
		\item[Cutset]: define a \textit{cutset} to be a set whose removal leaves the root $\rho$ in a finite component. \textit{Branching number ${\sf br}(T)$} of a tree $T$ can be defined as ${\rm br}(T) = \inf\left\{ \lambda>0: ~~\inf\limits_{\text{cutset}~C} \sum_{x \in C} \lambda^{-|x|} = 0 \right\}$.
	\end{itemize}

	Let us fix a particular node $\rho \in V^{\rm u}$ in p-SBM and focus on its depth-$t$ local tree $T_{\leq t}(\rho)$.  For $T_{\leq t}(\rho)$, denote the number of its labeled children at depth $i$ as $N_i(\rho)$. Consider the case $(1-\delta)\theta^2d>1$.
	Exactly as the method in Proof of Theorem~\ref{thm:AMP.k=2}, we have the following recursion for cumulant-generating function $K(\lambda)$ (where the expectation is taken over the label broadcasting process, conditionally on the Galton-Watson tree structure)
	\begin{align*}
	K_{ M_{t}(\rho) }(\lambda) &= K_{ M_{1}(\rho) }(\lambda) + \sum_{v \in \C^{\rm u}(\rho)} K_{ M_{t-1}(v) }(\theta \lambda) 
	\end{align*}
	which implies the following mean $\mu_t(\rho)$ and deviation $\sigma_t^2(\rho)$ bounds for message $M_{t}(\rho)$
	\begin{align*}
	\mu_t(\rho) &= N_1(\rho) \cdot \theta \log \frac{1+\theta}{1-\theta} + \theta^2 \cdot \sum_{v \in \C^{\rm u}(\rho)} \mu_{t-1}(v) ,\\
	\sigma_t^2(\rho) &= N_1(\rho) \cdot  \log^2 \frac{1+\theta}{1-\theta} + \theta^2 \cdot \sum_{v \in \C^{\rm u}(\rho)} [\mu_{t-1}^2(v) + \sigma_{t-1}^2(v)].
	\end{align*}
	Now we can easily find the following expression via the above equation
	\begin{align*}
	\mu_t(\rho) &= \left [ \sum_{i=1}^t \theta^{2(i-1)} N_i(\rho)   \right] \cdot \theta \log \frac{1+\theta}{1-\theta},\\
	\sigma_t^2(\rho) &= \left [ \sum_{i=1}^t \theta^{2(i-1)} N_i(\rho)   \right] \cdot  \log^2 \frac{1+\theta}{1-\theta} + \sum_{i=1}^{t-1}\theta^{2i} \sum_{v \in \C^{(i)}(\rho)} \mu^2_{t-i}(v).
	\end{align*}
	
	For the Galton-Watson tree (with ${\sf Poisson}((1-\delta)d$) off-spring distribution), we have $N_i(\rho)$ with growth rate $\asymp \delta d [(1-\delta)d]^{i-1}$, due to Kesten-Stigum Theorem  \citep{lyons2005probability}. Moreover, it can be shown that $\mu_t(\rho) \asymp [(1-\delta)\theta^2 d]^t$, as follows. Recall the flow definition of branching number for Galton-Watson tree, as the maximum $\lambda$ such that it admits a positive ${\sf flow}_{\lambda}(T)$. Thus the following representation of $\mu_t$ in terms of flow holds for any $\lambda < {\rm br}(T) =  (1-\delta)d$
	\begin{align}
	\mu_t(\rho) &= \left [ \sum_{i=1}^t [\theta^2 \lambda]^{i-1} \cdot \left(|N_{i}(\rho)|  \lambda^{-(i-1)}  \right) \right] \cdot \theta  \log \frac{1+\theta}{1-\theta} \nonumber \\
	&\geq  \left [ \sum_{i=1}^t [\theta^2 \lambda]^{i-1} {\sf flow}_{\lambda}(T)  \right] \cdot \delta d \theta \log \frac{1+\theta}{1-\theta} \nonumber \\
	& =\frac{ [\theta^2 \lambda]^t - 1}{\theta^2 \lambda - 1} {\sf flow}_{\lambda}(T) \cdot \delta d \theta \log \frac{1+\theta}{1-\theta}  = \Omega([\theta^2 \lambda]^t) 	\label{eq:water.flow}
	\end{align}
	due to the fact $|N_{i}(\rho)|  \lambda^{-(i-1)}  \geq \delta d \cdot {\sf flow}_{\lambda}(T)$, for any layer $i-1$. Taking $\lambda \uparrow (1-\delta)d$ ensures us that $ \mu_t(\rho) \asymp [(1-\delta)\theta^2 d]^t$ holds almost surely for Galton-Watson tree. 
	
	Now let us bound $\sigma^2_t/\mu^2_t$. In the regular tree case, we have shown $\lim_{t\rightarrow \infty} \frac{\sigma^2_t(\rho)}{\mu_t^2(\rho)} = \frac{1}{(1-\delta) \theta^2 d - 1}$. Here we want to show $\lim_{t\rightarrow \infty} \frac{\sigma^2_t(\rho)}{\mu_t^2(\rho)} \leq C \cdot \frac{1}{(1-\delta) \theta^2 d - 1}$. Conductance is a positive function ${\sf cond}(e)$ on the edges of $T$. Recall the energy definition ${\sf enrg}(T) := \sum_{e} [{\sf flow}(e)]^2/{\sf cond}(e)$. In addition to the earlier definitions, the branching number is the largest $\lambda$ such that the electric current flows with finite energy ${\sf enrg}_\lambda(T)$, given $\lambda^{-n}$ is the conductance of edges ${\sf cond}(e)$ at distance $n$ from the root of $T$. 
	For any $\lambda<{\rm br}(T) =(1-\delta)d$, we have
	\begin{align}
	\label{eq:elect.flow}
	\frac{\sigma^2_t(\rho)}{\mu_t^2(\rho)} & =  \frac{1}{\theta^2\delta d} \cdot \epsilon(t) + \sum_{i=1}^{t-1} [\theta^2 \lambda]^{-i} \cdot  \xi(i,t), \quad \text{where} \nonumber\\
	\epsilon(t) & := \frac{1}{\sum\limits_{i=1}^t \theta^{2(i-1)} N_i(\rho)  } \asymp \frac{1}{[(1-\delta)\theta^2d]^t}=o_t(1) \quad \text{due to equation~\eqref{eq:water.flow}} \\
	\xi(i,t) & := \frac{ \sum\limits_{v \in \C^{(i)}(\rho)} [\theta^{2i} \mu_{t-i}(v)]^2 }{\mu_t^2(\rho)}  \cdot \frac{1}{\lambda^{-i}} \leq \sum_{v \in \C^{(i)}(\rho)} [{\sf flow}(v)]^2\cdot \frac{1}{{\sf cond}(v)} < {\sf enrg}_{\lambda}(T) <\infty \label{eq:xi}\\
	\text{thus} ~~&\sum_{i=1}^{t-1} [\theta^2 \lambda]^{-i} \cdot  \xi(i,t) \leq \frac{1}{\theta^2\lambda -1} {\sf enrg}_{\lambda}(T) \nonumber.
	\end{align}
	Equation~\eqref{eq:xi} is due to the electric current ${\sf flow}(v) = \frac{\theta^{2i} \mu_{t-i}(v)}{\sum_{v \in \C^{(i)}(\rho)} \theta^{2i} \mu_{t-i}(v)}\geq \frac{\theta^{2i} \mu_{t-i}(v)}{\mu_t(\rho)}$. It is easy to verify that this flow satisfies the definition and that $\sum_{v \in \C^{(i)}(\rho)} {\sf flow}(v) = 1$ is the unit flow. In view of  \eqref{eq:elect.flow}, we have almost surely
	$$
	\lim_{t\rightarrow \infty} \frac{\sigma^2_t(\rho)}{\mu_t^2(\rho)} \leq C \cdot \frac{1}{(1-\delta) \theta^2 d - 1}.
	$$	
For a regular tree, we have $C \equiv 1$. In summary, conditionally on non-extinction, label recovery succeeds with probability at least $1 - \exp\left(  \frac{{\sf SNR}-1}{2C(1 + o(1))}\right)$. This establishes the upper bound.
	
	Now let us prove the lower bound. Consider the case $(1-\delta)\theta^2d<1$. Recall the proof of Theorem~\ref{thm:low.k=2} and define $$D_{T_{\leq t}(\rho)} := \max \left\{ d_{\chi^2} \left( \mu_{\ell_{T_{\leq t}(\rho)}}^{+},\mu_{\ell_{T_{\leq t}(\rho)}}^{-} \right) , d_{\chi^2} \left( \mu_{\ell_{T_{\leq t}(\rho)}}^{-},\mu_{\ell_{T_{\leq t}(\rho)}}^{+} \right)  \right\}$$ (abbreviate as $D(\rho)$ when there is no confusion), we have the following recursion
	\begin{align*}
	\log (1+D_{T_{\leq t}(\rho)} )   \leq   N_1(\rho) \cdot \log \left(1+\frac{4\theta^2}{1-\theta^2}\right) + \theta^2 \sum_{v \in \C^{\rm u}(\rho)} \frac{\log (1+\theta^2 \cdot D_{T_{\leq t-1}(v)})}{\theta^2}.
	\end{align*}	
	Invoke the following fact,
	\begin{align*}
		\frac{\log (1+\theta^2 x)}{\theta^2} \leq (1+\eta) \log(1+x) \quad \text{for all}~~0 \leq x \leq \eta,~\forall \theta,
	\end{align*}
	whose proof is in one line
	$$
	\frac{\log (1+\theta^2 x)}{\theta^2} \leq x \leq (1+\eta) \frac{x}{1+x} \leq (1+\eta) \log(1+x).
	$$
	Thus if $D_{t-1} \leq \eta$, then the following holds
	\begin{align}
		\label{eq:low.recusion}
		\log (1+D_{T_{\leq t}(\rho)} )   \leq   N_1(\rho) \cdot \log \left(1+\frac{4\theta^2}{1-\theta^2}\right) + (1+\eta)\theta^2 \sum_{v \in \C^{\rm u}(\rho)} \log (1+ D_{T_{\leq t-1}(v)}).
	\end{align}
	Denoting
	$$
	d_{T_{\leq t}(\rho)}   := \log (1+D_{T_{\leq t}(\rho)} ) ,
	$$
	Equation~\eqref{eq:low.recusion} becomes
	$$
	d_{T_{\leq t}(\rho)} \leq  N_1(\rho) \cdot \log \left(1+\frac{4\theta^2}{1-\theta^2}\right) + (1+\eta)\theta^2 \sum_{v \in \C^{\rm u}(\rho)} d_{T_{\leq t-1}(v)}.
	$$
	
	We will need the definition of branching number ${\rm br}(T)  = (1-\delta)d$ via cutset.
	\begin{lemma}[\cite{pemantle1999robust}, Lemma 3.3]
		Assume $ {\rm br}(T) < \lambda$. Then for all $\epsilon>0$, there exists a cutset $C$ such that
		\begin{align}
			\label{eq:cutset.all}
		\sum_{x \in C} \left( \frac{1}{\lambda} \right)^{|x|} \leq \epsilon
		\end{align}
		and for all $v$ such that $|v| \leq \max_{x \in C} |x|$,
		\begin{align}
			\label{eq:cutset.partial}
		\sum_{x \in C \cap T(v)} \left( \frac{1}{\lambda} \right)^{|x|-|v|} \leq 1.
		\end{align}
		Here the notation $|v|$ denotes the depth of $v$.
	\end{lemma}
	Fix any $\lambda$ such that$\frac{1}{\theta^2}> \lambda >{\rm br}(T)  = (1-\delta)d$ (this is doable because $(1-\delta)\theta^2d<1$). 
	Define function $g_{\alpha}(\eta) = \frac{\eta}{1+\eta}[1-(1+\eta)\alpha]$ for $\alpha<1$, clearly it is a monotone increasing function in $\eta <\sqrt{1/\alpha}-1$. Thus the inverse $g^{-1}_{\alpha}(y)$ exists if $
	y < g_{\alpha}(\sqrt{1/\alpha}-1)  = (1-\sqrt{\alpha})^2.
	$
	Under the assumption
	\begin{align}
	 \delta d \cdot  \log \left(1+\frac{4\theta^2}{1-\theta^2}\right) < (1-\sqrt{\theta^2 \lambda})^2,
	\end{align}
	Choose
	\begin{align*}
	\eta = g^{-1}_{\theta^2 \lambda} \left(  \delta d \cdot  \log \left(1+\frac{4\theta^2}{1-\theta^2}\right) \right)  ~~\text{which implies}~~ \delta d \cdot  \log \left(1+\frac{4\theta^2}{1-\theta^2}\right) = g_{\theta^2 \lambda}  (\eta)
	\end{align*}
	and we know $\eta <\sqrt{1/\theta^2\lambda}-1 \Rightarrow (1+\eta)\theta^2\lambda<1$.
	The reason will be clear in a second.
	
	For any $\epsilon$ small, the above Lemma claims the existence of cutset $C_\epsilon$ such that equations~\eqref{eq:cutset.all} and\eqref{eq:cutset.partial} holds.
	 Let's prove through induction on $\max_{x \in C_\epsilon} |x| - |v|$ that for any $v$ such that $|v| \leq \max_{x \in C_\epsilon} |x|$, we have
	\begin{align}
		\label{eq:induction}
		d_{T_{\leq C_{\epsilon}}(v)} \leq \frac{\eta}{1+\eta} \sum_{x \in C_{\epsilon} \cap T(v)} \left( \frac{1}{\lambda} \right)^{|x|-|v|} \leq \frac{\eta}{1+\eta} .
	\end{align}
	Note for the start of induction $v \in C_{\epsilon}$, $$d_{T_{\leq C_{\epsilon}}(v)} = \delta d \cdot  \log \left(1+\frac{4\theta^2}{1-\theta^2}\right) = g_{\theta^2 \lambda}(\eta) \leq \frac{\eta}{1+\eta} (1 - (1+\eta) \theta^2 \lambda) <  \frac{\eta}{1+\eta}.$$ Now precede with the induction, assume for $v$ such that $\max_{x \in C_\epsilon} |x| - |v| = t-1$ equation~\eqref{eq:induction} is satisfied, let's prove for $\rho: \max_{x \in C_\epsilon} |x| - |\rho| = t$.  Due to the fact for all $v \in \C^{\rm u}(\rho)$, $d_{T_{\leq C_{\epsilon}}(v)}  \leq \frac{\eta}{1+\eta} \Rightarrow D_{T_{\leq C_{\epsilon}}(v)}  \leq \eta$, we can recall the linearized recursion
	\begin{align*}
		d_{T_{\leq C_{\epsilon}}(\rho)} &\leq  N_1(\rho) \cdot \log \left(1+\frac{4\theta^2}{1-\theta^2}\right) + (1+\eta)\theta^2 \sum_{v \in \C^{\rm u}(\rho)} d_{T_{\leq C_{\epsilon}}(v)} \\
		& \leq N_1(\rho) \cdot \log \left(1+\frac{4\theta^2}{1-\theta^2}\right) + (1+\eta)\theta^2 \sum_{v \in \C^{\rm u}(\rho)} \left[  \frac{\eta}{1+\eta} \sum_{x \in C_{\epsilon} \cap T(v)} \left( \frac{1}{\lambda} \right)^{|x|-|v|} \right] \\
		& \leq N_1(\rho) \cdot \log \left(1+\frac{4\theta^2}{1-\theta^2}\right) + \frac{\eta}{1+\eta} \cdot (1+\eta)\theta^2\lambda  \sum_{v \in \C^{\rm u}(\rho)} \sum_{x \in C_{\epsilon} \cap T(v)} \left( \frac{1}{\lambda} \right)^{|x|-|v|-1}  \\
		& \leq N_1(\rho) \cdot \log \left(1+\frac{4\theta^2}{1-\theta^2}\right) + \frac{\eta}{1+\eta} \cdot  (1+\eta)\theta^2\lambda  \sum_{x \in C_{\epsilon} \cap T(\rho)} \left( \frac{1}{\lambda} \right)^{|x|-|\rho|}  \\
		& \leq \frac{\eta}{1+\eta} (1 - (1+\eta) \theta^2 \lambda)+\frac{\eta}{1+\eta} \cdot  (1+\eta)\theta^2\lambda \leq \frac{\eta}{1+\eta}.
	\end{align*}
So far we have proved for any $v$, such that $|v| \leq \max_{x \in C_\epsilon} |x|$
\begin{align*}
	d_{T_{\leq C_{\epsilon}}(v)} \leq \frac{\eta}{1+\eta} \sum_{x \in C_{\epsilon} \cap T(v)} \left( \frac{1}{\lambda} \right)^{|x|-|v|} \leq \frac{\eta}{1+\eta}\\
	\text{which implies}\quad D_{T_{\leq C_{\epsilon}}(v)} \leq \eta
\end{align*}
so that the linearized recursion~\eqref{eq:low.recusion} always holds.
	Take $\epsilon \rightarrow 0, \lambda \rightarrow  (1-\delta)d$. Define $t_\epsilon: = \min\{ |x|, x\in C_{\epsilon} \}$, it is also easy to see from equation~\eqref{eq:cutset.all} that 
	$$
	\left( \frac{1}{\lambda} \right)^{t_\epsilon} \leq \sum_{x \in C_\epsilon} \left( \frac{1}{\lambda} \right)^{|x|} \leq \epsilon \Rightarrow t_{\epsilon} > \frac{\log(1/\epsilon)}{\log \lambda} \rightarrow \infty.
	$$
	Putting things together, under the condition
	\begin{align*} 
		\delta d \cdot  \log \left(1+\frac{4\theta^2}{1-\theta^2}\right) \leq (1-\sqrt{(1-\delta)\theta^2 d})^2,
	\end{align*} we have $$\lim_{t \rightarrow \infty} D_{T_{\leq t}(\rho)} = \lim_{\epsilon \rightarrow 0} D_{T_{\leq C_\epsilon}(\rho) }\leq \eta \leq C \cdot \frac{\delta d \log \left(1+\frac{4\theta^2}{1-\theta^2}\right) }{ 1- (1 -\delta)\theta^2 d}.$$
	Here the last step is due to a simple bound on $\eta$ based on the inequality
	$$
	\delta d \cdot  \log \left(1+\frac{4\theta^2}{1-\theta^2}\right) = \frac{\eta}{1+\eta} - \eta \cdot (1-\delta)\theta^2 d > \eta[1 - (1-\delta)\theta^2 d] -\eta^2.
	$$
\end{proof}

\bigskip

\begin{proof}[Proof of Theorem~\ref{thm:k-AMP}]
	
	For $\alpha > 1$:\\
	Use induction analysis. For $t=1$, the result follows from Hoeffding's lemma. Assume results hold for $t-1$, then if above the fraction label is $l$,
	\begin{align*}
		&\quad \mathbb{E}\left[ e^{\lambda M_{t}(\ell_{T_{\leq t}(v)})} | \ell(v) = l \right] \\
		& \leq e^{\frac{\lambda^2}{2} \sigma_1^2} \cdot e^{\lambda \mu_1} \cdot \prod_{u \in \C^{\rm u}(v)}\mathbb{E}\left[ e^{\lambda \theta M_{t-1}(\ell_{T_{\leq t-1}(u)})} | \ell(v) = l \right] \\
		& = e^{\frac{\lambda^2}{2} \sigma_{1}^2 } \cdot e^{\lambda \mu_1} \cdot \prod_{u \in \C^{\rm u}(v)} \left[ \left(\theta + \frac{1-\theta}{k}\right) \cdot e^{\lambda \theta \mu_{t-1}} e^{\frac{\lambda^2 \theta^2 \sigma_{t-1}^2}{2}} + \frac{1-\theta}{k} \cdot e^{-\lambda \theta \mu_{t-1}} e^{\frac{\lambda^2 \theta^2 \sigma_{t-1}^2}{2}} + \frac{(k-2)(1-\theta)}{k} \cdot e^{\frac{\lambda^2 \theta^2 \sigma_{t-1}^2}{2}} \right] \\
		& \leq e^{\lambda \left(\mu_1 + \alpha \mu_{t-1}\right)} \cdot e^{\frac{\lambda^2}{2} \cdot \left( \sigma_1^2 + (1-\delta)d \theta^2 \sigma_{t-1}^2 \right) } \cdot e^{\frac{\lambda^2}{2}\cdot (1-\delta)d\theta^2 \mu_{t-1}^2}\\
		& \leq e^{\lambda \left(\mu_1 + \alpha \mu_{t-1}\right)} \cdot e^{\frac{\lambda^2}{2} \cdot \left( \sigma_1^2 + (1-\delta)d \theta^2 \sigma_{t-1}^2 +(1-\delta)d\theta^2 \mu_{t-1}^2 \right) }
	\end{align*}
	where the last step uses Hoeffding's Lemma~\ref{lma:hoeff}. When none of the labels is $l$, we have the following bound
	\begin{align*}
		&\quad \mathbb{E}\left[ e^{\lambda M_{t}(\ell_{T_{\leq t}(v)})} | \ell(v) = l \right] \\
		&\leq e^{\frac{\lambda^2}{2} \sigma_{1}^2 } \cdot \prod_{u \in \C^{\rm u}(v)} \left[ \frac{1-\theta}{k} \cdot e^{\lambda \theta \mu_{t-1}} e^{\frac{\lambda^2 \theta^2 \sigma_{t-1}^2}{2}} + \frac{1-\theta}{k} \cdot e^{-\lambda \theta \mu_{t-1}} e^{\frac{\lambda^2 \theta^2 \sigma_{t-1}^2}{2}} +\left( \theta + \frac{(k-2)(1-\theta)}{k} \right) \cdot e^{\frac{\lambda^2 \theta^2 \sigma_{t-1}^2}{2}} \right] \\
		& \leq e^{\frac{\lambda^2}{2} \cdot \left( \sigma_1^2 + (1-\delta)d \theta^2 \sigma_{t-1}^2 \right) } \cdot e^{\frac{\lambda^2}{2}\cdot (1-\delta)d\theta^2 \mu_{t-1}^2}.
	\end{align*}
	Proof is completed.
\end{proof}

\bigskip

\begin{proof}[Proof of Theorem~\ref{thm:low.g-k}]

Borrowing the idea from Proof~\ref{pf:low.k=2}, we can study the following testing problem:
\begin{align*}
&d_{\chi^2} \left( \mu_{\ell_{T_{\leq t}(\rho)}}^{(i)},\mu_{\ell_{T_{\leq t}(\rho)}}^{(j)} \right)  \\
&= \left( 1 + \theta^2\left( \frac{1}{\theta+\frac{1-\theta}{k}} + \frac{1}{\frac{1-\theta}{k}} \right) \right)^{\delta d} \left[ 1 + d_{\chi^2} \left( \theta \mu_{\ell_{\leq t-1}(v)}^{(i)} + (1-\theta) \bar{\mu}_{\ell_{\leq t-1}(v)} , \theta \mu_{\ell_{\leq t-1}(v)}^{(j)} + (1-\theta) \bar{\mu}_{\ell_{\leq t-1}(v)} \right) \right]^{(1-\delta)d} - 1
\end{align*}

We know
\begin{align*}
	& d_{\chi^2} \left( \theta \mu_{\ell_{\leq t-1}(v)}^{(i)} + (1-\theta) \bar{\mu}_{\ell_{\leq t-1}(v)} , \theta \mu_{\ell_{\leq t-1}(v)}^{(j)} + (1-\theta) \bar{\mu}_{\ell_{\leq t-1}(v)} \right)  = \int \frac{\theta^2 (\mu_{\ell_{\leq t-1}(v)}^{(i)} - \mu_{\ell_{\leq t-1}(v)}^{(j)})^2}{\theta \mu_{\ell_{\leq t-1}(v)}^{(j)} + (1-\theta) \bar{\mu}_{\ell_{\leq t-1}(v)}} dx \\
	& \leq  \theta^2 \left[ (\theta + \frac{1 - \theta}{k}) d_{\chi^2} \left( \mu_{\ell_{\leq t-1}(v)}^{(i)},\mu_{\ell_{\leq t-1}(v)}^{(j)} \right) + \frac{1-\theta}{k} d_{\chi^2} \left( \mu_{\ell_{\leq t-1}(v)}^{(j)},\mu_{\ell_{\leq t-1}(v)}^{(i)} \right) \right. \\
	&\left. + \frac{1-\theta}{k} \sum_{l \in [k]\backslash \{i,j\}} 2 \left( d_{\chi^2} \left( \mu_{\ell_{\leq t-1}(v)}^{(i)},\mu_{\ell_{\leq t-1}(v)}^{(l)} \right) + d_{\chi^2} \left(\mu_{\ell_{\leq t-1}(v)}^{(j)},\mu_{\ell_{\leq t-1}(v)}^{(l)} \right)  \right)    \right]\\
	& \leq \theta^2 (1 + \frac{3(1-\theta)(k-2)}{k}) \cdot d_{\chi^2}^{t-1}
\end{align*}
Thus define
$$
d_{\chi^2}^{t} := \max_{i,j \in [k], i\neq j} d_{\chi^2} \left( \mu_{\ell_{T_{\leq t}(\rho)}}^{(i)},\mu_{\ell_{T_{\leq t}(\rho)}}^{(j)} \right)
$$
Then
\begin{align*}
	\log(1 + d_{\chi^2}^{t} )\leq \delta d \cdot \log \left( 1 + \theta^2\left( \frac{1}{\theta+\frac{1-\theta}{k}} + \frac{1}{\frac{1-\theta}{k}} \right) \right) + (1-\delta)d  \cdot \log \left(1 +  \theta^2  (1 + \frac{3(1-\theta)(k-2)}{k}) \cdot d_{\chi^2}^{t-1} \right)
\end{align*}

Thus if
$$
(1-\delta)\theta^2d (1 + \frac{3(1-\theta)(k-2)}{k})  < 1,
$$
denote $c^*$ as the fixed point of the equation
$$
\log(1 + c^* ) = \delta d \cdot \log \left( 1 + \theta^2\left( \frac{1}{\theta+\frac{1-\theta}{k}} + \frac{1}{\frac{1-\theta}{k}} \right) \right) + (1-\delta)d  \cdot \log \left(1 +  \theta^2  (1 + \frac{3(1-\theta)(k-2)}{k}) \cdot c^* \right).
$$
We have the following upper bounds for $c^*$ via the fact that $x - \frac{1}{2}x^2 < \log(1+x) < x$
$$
c^* - \frac{1}{2}(c^*)^2 \leq  \delta d \cdot \log \left( 1 + \theta^2\left( \frac{1}{\theta+\frac{1-\theta}{k}} + \frac{1}{\frac{1-\theta}{k}} \right) \right) + (1-\delta)\theta^2d (1 + \frac{3(1-\theta)(k-2)}{k})     \cdot c^* .
$$
The above equation implies $c^* < \frac{ 2\delta d \cdot \log \left( 1 + \theta^2\left( \frac{1}{\theta+\frac{1-\theta}{k}} + \frac{1}{\frac{1-\theta}{k}} \right) \right) }{ 1 - (1-\delta)\theta^2d (1 + \frac{3(1-\theta)(k-2)}{k}) },$
and
\begin{align*}
	\log(1 + d_{\chi^2}^{t} ) \leq \frac{ 2\delta d \cdot \log \left( 1 + \theta^2\left( \frac{1}{\theta+\frac{1-\theta}{k}} + \frac{1}{\frac{1-\theta}{k}} \right) \right) }{ 1 - (1-\delta)\theta^2d (1 + \frac{3(1-\theta)(k-2)}{k}) }.
\end{align*}

Invoke the following Lemma from \cite{tsybakov2009introduction}'s  Proposition 2.4.
\begin{lemma}[\cite{tsybakov2009introduction}, Proposition 2.4]
	\label{lma:multi-test}
	Let $P_0,P_1,\ldots, P_{k-1}$ be probability measures on $(\mathcal{X}, \mathcal{A})$ satisfying 
	$$
	\frac{1}{k-1}\sum_{i=1}^{k-1} d_{\chi^2}(P_j, P_0) \leq (k-1) \cdot \alpha_*
	$$
	then we have for any selector $\psi:\mathcal{X} \rightarrow [k]$
	$$
	\max_{i \in [k]} P_i(\psi \neq i) \geq \frac{1}{2}[1 - \alpha_* - \frac{1}{k-1}]
	$$
\end{lemma}
Since we have $\frac{1}{k-1} \sum_{i\in [k]\backslash j } d_{\chi^2}  \left( \mu_{\ell_{T_{\leq t}(\rho)}}^{(i)},\mu_{\ell_{T_{\leq t}(\rho)}}^{(j)} \right) \leq \alpha \cdot (k-1) $, 
we apply Lemma~\ref{lma:multi-test} and obtain
$$
\inf_{\Phi} \sup_{l \in [k]} \mathbb{P}\left( \Phi \neq l \right) \geq \frac{1}{2} \left( 1 - \alpha - \frac{1}{k-1} \right),
$$
where $\alpha = \frac{\delta}{1-\delta}  \cdot \frac{ {\sf SNR} }{ 1 - 4 \cdot {\sf SNR} } \cdot \frac{2(p+q)(q+p/(k-1))}{pq}$.

\end{proof}

\bigskip
\begin{proof}[Proof of Theorem~\ref{thm:sbm.g-k}]
	The proof of Theorem~\ref{thm:sbm.g-k} is the same idea as the proof of Theorem~\ref{thm:sbm.k=2}, with Theorem~\ref{thm:k-AMP} and Theorem~\ref{thm:low.g-k} as the case for regular tree. For simplicity, we will not prove it again.
\end{proof}

\section*{Acknowledgements}
The authors want to thank Elchanan Mossel for many valuable discussions.

\bibliographystyle{plainnat} 
\bibliography{bibfile}

\begin{thebibliography}{35}
\providecommand{\natexlab}[1]{#1}
\providecommand{\url}[1]{\texttt{#1}}
\expandafter\ifx\csname urlstyle\endcsname\relax
  \providecommand{\doi}[1]{doi: #1}\else
  \providecommand{\doi}{doi: \begingroup \urlstyle{rm}\Url}\fi

\bibitem[Abbe and Sandon(2015{\natexlab{a}})]{abbe2015community}
Emmanuel Abbe and Colin Sandon.
\newblock Community detection in general stochastic block models: fundamental
  limits and efficient recovery algorithms.
\newblock \emph{arXiv preprint arXiv:1503.00609}, 2015{\natexlab{a}}.

\bibitem[Abbe and Sandon(2015{\natexlab{b}})]{abbe2015detection}
Emmanuel Abbe and Colin Sandon.
\newblock Detection in the stochastic block model with multiple clusters: proof
  of the achievability conjectures, acyclic bp, and the information-computation
  gap.
\newblock \emph{arXiv preprint arXiv:1512.09080}, 2015{\natexlab{b}}.

\bibitem[Abbe et~al.(2014)Abbe, Bandeira, and Hall]{abbe2014exact}
Emmanuel Abbe, Afonso~S Bandeira, and Georgina Hall.
\newblock Exact recovery in the stochastic block model.
\newblock \emph{arXiv preprint arXiv:1405.3267}, 2014.

\bibitem[Adamic and Glance(2005)]{adamic2005political}
Lada~A Adamic and Natalie Glance.
\newblock The political blogosphere and the 2004 us election: divided they
  blog.
\newblock In \emph{Proceedings of the 3rd international workshop on Link
  discovery}, pages 36--43. ACM, 2005.

\bibitem[Chen and Xu(2014)]{chen2014statistical}
Yudong Chen and Jiaming Xu.
\newblock Statistical-computational tradeoffs in planted problems and submatrix
  localization with a growing number of clusters and submatrices.
\newblock \emph{arXiv preprint arXiv:1402.1267}, 2014.

\bibitem[Coja-Oghlan(2010)]{coja2010graph}
Amin Coja-Oghlan.
\newblock Graph partitioning via adaptive spectral techniques.
\newblock \emph{Combinatorics, Probability and Computing}, 19\penalty0
  (02):\penalty0 227--284, 2010.

\bibitem[Cucuringu et~al.(2012)Cucuringu, Singer, and
  Cowburn]{cucuringu2012eigenvector}
Mihai Cucuringu, Amit Singer, and David Cowburn.
\newblock Eigenvector synchronization, graph rigidity and the molecule problem.
\newblock \emph{Information and Inference}, 1\penalty0 (1):\penalty0 21--67,
  2012.

\bibitem[Decelle et~al.(2011)Decelle, Krzakala, Moore, and
  Zdeborov{\'a}]{decelle2011asymptotic}
Aurelien Decelle, Florent Krzakala, Cristopher Moore, and Lenka Zdeborov{\'a}.
\newblock Asymptotic analysis of the stochastic block model for modular
  networks and its algorithmic applications.
\newblock \emph{Physical Review E}, 84\penalty0 (6):\penalty0 066106, 2011.

\bibitem[Deshpande et~al.(2015)Deshpande, Abbe, and
  Montanari]{deshpande2015asymptotic}
Yash Deshpande, Emmanuel Abbe, and Andrea Montanari.
\newblock Asymptotic mutual information for the two-groups stochastic block
  model.
\newblock \emph{arXiv preprint arXiv:1507.08685}, 2015.

\bibitem[Evans et~al.(2000)Evans, Kenyon, Peres, and
  Schulman]{evans2000broadcasting}
William Evans, Claire Kenyon, Yuval Peres, and Leonard~J Schulman.
\newblock Broadcasting on trees and the ising model.
\newblock \emph{Annals of Applied Probability}, pages 410--433, 2000.

\bibitem[Gamarnik and Sudan(2014)]{gamarnik2014limits}
David Gamarnik and Madhu Sudan.
\newblock Limits of local algorithms over sparse random graphs.
\newblock In \emph{Proceedings of the 5th conference on Innovations in
  theoretical computer science}, pages 369--376. ACM, 2014.

\bibitem[Gao et~al.(2015)Gao, Ma, Zhang, and Zhou]{gao2015achieving}
Chao Gao, Zongming Ma, Anderson~Y Zhang, and Harrison~H Zhou.
\newblock Achieving optimal misclassification proportion in stochastic block
  model.
\newblock \emph{arXiv preprint arXiv:1505.03772}, 2015.

\bibitem[Hajek et~al.(2014)Hajek, Wu, and Xu]{hajek2014achieving}
Bruce Hajek, Yihong Wu, and Jiaming Xu.
\newblock Achieving exact cluster recovery threshold via semidefinite
  programming.
\newblock \emph{arXiv preprint arXiv:1412.6156}, 2014.

\bibitem[Hajek et~al.(2015)Hajek, Wu, and Xu]{hajek2015achieving}
Bruce Hajek, Yihong Wu, and Jiaming Xu.
\newblock Achieving exact cluster recovery threshold via semidefinite
  programming: Extensions.
\newblock \emph{arXiv preprint arXiv:1502.07738}, 2015.

\bibitem[Janson and Mossel(2004)]{janson2004robust}
Svante Janson and Elchanan Mossel.
\newblock Robust reconstruction on trees is determined by the second
  eigenvalue.
\newblock \emph{Annals of probability}, pages 2630--2649, 2004.

\bibitem[Jin(2015)]{jin2015fast}
Jiashun Jin.
\newblock Fast community detection by score.
\newblock \emph{The Annals of Statistics}, 43\penalty0 (1):\penalty0 57--89,
  2015.

\bibitem[Kanade et~al.(2014)Kanade, Mossel, and Schramm]{kanade2014global}
Varun Kanade, Elchanan Mossel, and Tselil Schramm.
\newblock Global and local information in clustering labeled block models.
\newblock \emph{arXiv preprint arXiv:1404.6325}, 2014.

\bibitem[Kesten and Stigum(1966{\natexlab{a}})]{kesten1966additional}
Harry Kesten and Bernt~P Stigum.
\newblock Additional limit theorems for indecomposable multidimensional
  galton-watson processes.
\newblock \emph{The Annals of Mathematical Statistics}, pages 1463--1481,
  1966{\natexlab{a}}.

\bibitem[Kesten and Stigum(1966{\natexlab{b}})]{kesten1966limit}
Harry Kesten and Bernt~P Stigum.
\newblock A limit theorem for multidimensional galton-watson processes.
\newblock \emph{The Annals of Mathematical Statistics}, 37\penalty0
  (5):\penalty0 1211--1223, 1966{\natexlab{b}}.

\bibitem[Kleinberg(2000)]{kleinberg2000small}
Jon Kleinberg.
\newblock The small-world phenomenon: An algorithmic perspective.
\newblock In \emph{Proceedings of the thirty-second annual ACM symposium on
  Theory of computing}, pages 163--170. ACM, 2000.

\bibitem[Krzakala et~al.(2013)Krzakala, Moore, Mossel, Neeman, Sly,
  Zdeborov{\'a}, and Zhang]{krzakala2013spectral}
Florent Krzakala, Cristopher Moore, Elchanan Mossel, Joe Neeman, Allan Sly,
  Lenka Zdeborov{\'a}, and Pan Zhang.
\newblock Spectral redemption in clustering sparse networks.
\newblock \emph{Proceedings of the National Academy of Sciences}, 110\penalty0
  (52):\penalty0 20935--20940, 2013.

\bibitem[Linial(1992)]{linial1992locality}
Nathan Linial.
\newblock Locality in distributed graph algorithms.
\newblock \emph{SIAM Journal on Computing}, 21\penalty0 (1):\penalty0 193--201,
  1992.

\bibitem[Lyons and Peres(2005)]{lyons2005probability}
Russell Lyons and Yuval Peres.
\newblock Probability on trees and networks, 2005.

\bibitem[Massouli{\'e}(2014)]{massoulie2014community}
Laurent Massouli{\'e}.
\newblock Community detection thresholds and the weak ramanujan property.
\newblock In \emph{Proceedings of the 46th Annual ACM Symposium on Theory of
  Computing}, pages 694--703. ACM, 2014.

\bibitem[Mossel(2001)]{mossel2001reconstruction}
Elchanan Mossel.
\newblock Reconstruction on trees: beating the second eigenvalue.
\newblock \emph{Annals of Applied Probability}, pages 285--300, 2001.

\bibitem[Mossel and Peres(2003)]{mossel2003information}
Elchanan Mossel and Yuval Peres.
\newblock Information flow on trees.
\newblock \emph{The Annals of Applied Probability}, 13\penalty0 (3):\penalty0
  817--844, 2003.

\bibitem[Mossel et~al.(2012)Mossel, Neeman, and Sly]{mossel2012stochastic}
Elchanan Mossel, Joe Neeman, and Allan Sly.
\newblock Stochastic block models and reconstruction.
\newblock \emph{arXiv preprint arXiv:1202.1499}, 2012.

\bibitem[Mossel et~al.(2013{\natexlab{a}})Mossel, Neeman, and
  Sly]{mossel2013belief}
Elchanan Mossel, Joe Neeman, and Allan Sly.
\newblock Belief propagation, robust reconstruction, and optimal recovery of
  block models.
\newblock \emph{arXiv preprint arXiv:1309.1380}, 2013{\natexlab{a}}.

\bibitem[Mossel et~al.(2013{\natexlab{b}})Mossel, Neeman, and
  Sly]{mossel2013proof}
Elchanan Mossel, Joe Neeman, and Allan Sly.
\newblock A proof of the block model threshold conjecture.
\newblock \emph{arXiv preprint arXiv:1311.4115}, 2013{\natexlab{b}}.

\bibitem[Nguyen and Onak(2008)]{nguyen2008constant}
Huy~N Nguyen and Krzysztof Onak.
\newblock Constant-time approximation algorithms via local improvements.
\newblock In \emph{Foundations of Computer Science, 2008. FOCS'08. IEEE 49th
  Annual IEEE Symposium on}, pages 327--336. IEEE, 2008.

\bibitem[Parnas and Ron(2007)]{parnas2007approximating}
Michal Parnas and Dana Ron.
\newblock Approximating the minimum vertex cover in sublinear time and a
  connection to distributed algorithms.
\newblock \emph{Theoretical Computer Science}, 381\penalty0 (1):\penalty0
  183--196, 2007.

\bibitem[Pemantle and Steif(1999)]{pemantle1999robust}
Robin Pemantle and Jeffrey~E Steif.
\newblock Robust phase transitions for heisenberg and other models on general
  trees.
\newblock \emph{Annals of Probability}, pages 876--912, 1999.

\bibitem[Tsybakov(2009)]{tsybakov2009introduction}
Alexandre~B Tsybakov.
\newblock \emph{Introduction to nonparametric estimation}, volume~11.
\newblock Springer Series in Statistics, 2009.

\bibitem[Zhang and Zhou(2015)]{zhang2015minimax}
Anderson~Y Zhang and Harrison~H Zhou.
\newblock Minimax rates of community detection in stochastic block models.
\newblock \emph{arXiv preprint arXiv:1507.05313}, 2015.

\bibitem[Zhang et~al.(2014)Zhang, Moore, and Zdeborov{\'a}]{zhang2014phase}
Pan Zhang, Cristopher Moore, and Lenka Zdeborov{\'a}.
\newblock Phase transitions in semisupervised clustering of sparse networks.
\newblock \emph{Physical Review E}, 90\penalty0 (5):\penalty0 052802, 2014.

\end{thebibliography}

\end{document}